\documentclass[11pt,epsfig,amsfonts]{amsart}
\usepackage{stmaryrd}
\usepackage{xcolor}
\usepackage{epsfig}
\usepackage{enumerate}
\usepackage{amsmath}
\usepackage{amssymb}
\usepackage{amscd}
\usepackage{graphicx}
\usepackage{lineno}
\usepackage{pstricks}
\usepackage{tocvsec2}
\usepackage{mathtools}
\allowdisplaybreaks[4]

\usepackage [backref, colorlinks, citecolor=red, anchorcolor=black, linkcolor=blue]{hyperref}

\usepackage{amsfonts}
\usepackage[top=35mm, bottom=35mm, left=30mm, right=30mm]{geometry}
\usepackage{verbatim}
\usepackage{graphicx, amssymb, amsmath,amsthm}
\allowdisplaybreaks[4]
\usepackage{appendix}
\usepackage{mathrsfs}
\usepackage{cite}

\makeatletter
\@namedef{subjclassname@2020}{%
	\textup{2020} Mathematics Subject Classification}
\makeatother
\usepackage{xcolor}

\usepackage{xcolor}

\newtheorem{definition}{Definition}[section]
\newtheorem{lemma}{Lemma}[section]
\newtheorem*{question}{Question}%[section]
\newtheorem{theorem}{Theorem}%[section]
\newtheorem{corollary}{Corollary}%[section]
\newtheorem{remark}{Remark}[section]
\newtheorem{example}{Example}[section]
\newtheorem{proposition}{Proposition}[section]

\numberwithin{equation}{section}
\usepackage{epsfig}

\numberwithin{equation}{section}

\renewcommand*{\backref}[1]{}
\renewcommand*{\backrefalt}[4]{\quad \tiny
	\ifcase #1 (\textbf{NOT CITED.})%
	\or    (Cited on Section~#2.)%
	\else   (Cited on Section~#2.)%
	\fi}

\makeatletter

\def\MRbibitem{\@ifnextchar[\my@lbibitem\my@bibitem}

\def\mybiblabel#1#2{\@biblabel{{\hyperref{http://www.ams.org/mathscinet-getitem?mr=#1}{}{}{#2}}}}

\def\myhyperanchor#1{\Hy@raisedlink{\hyper@anchorstart{cite.#1}\hyper@anchorend}}

\def\my@lbibitem[#1]#2#3#4\par{%
	\item[\mybiblabel{#2}{#1}\myhyperanchor{#3}\hfill]#4%
	\@ifundefined{ifbackrefparscan}{}{\BR@backref{#3}}%
	\if@filesw{\let\protect\noexpand\immediate% write to aux-file
		\write\@auxout{\string\bibcite{#3}{#1}}}\fi\ignorespaces%
}

\def\my@bibitem#1#2#3\par{%
	\refstepcounter\@listctr% standard tex item counter for the generic item number
	\item[\mybiblabel{#1}{\the\value\@listctr}\myhyperanchor{#2}\hfill]#3%
	\@ifundefined{ifbackrefparscan}{}{\BR@backref{#2}}%
	\if@filesw\immediate\write\@auxout% write to aux-file
	{\string\bibcite{#2}{\the\value\@listctr}}\fi\ignorespaces%
}

\makeatother

\subjclass[2020]{Primary: 37A35, 37C45.  Second: 37B05.}
\keywords{dimension theory; multifractal analysis; almost weak specification property; metric mean dimension}

\author{Chunlin Liu}
\address[Chunlin Liu]
{CAS Wu Wen-Tsun Key Laboratory of Mathematics, School of Mathematical Sciences, University of Science and Technology of China, Hefei, Anhui, 230026, P.R. China}
\email[C.~Liu]{lcl666@mail.ustc.edu.cn}

\author{Xue Liu}
\address[Xue Liu]
{CAS Wu Wen-Tsun Key Laboratory of Mathematics, School of Mathematical Sciences, University of Science and Technology of China, Hefei, Anhui, 230026, P.R. China}
\email[X.~Liu]{xueliu21@ustc.edu.cn}

\begin{document}
	
	\begin{abstract}
		Let $(X, d)$ be a compact metric space, $f: X \to X$ be a continuous transformation with the almost weak specification property and $\varphi: X \to \mathbb{R}$ be a continuous function. We consider the set (called the irregular set for $\varphi$) of points for which the Birkhoff average of $\varphi$ does not exist and show that this set is either empty or carries full Bowen upper and lower metric mean dimension. %By the similar argument, we also provide a variational principle for the upper and lower metric mean dimension of level sets $\{x\in X:\lim_{n\to \infty}\frac{1}{n}\sum_{i=0}^{n-1}\varphi (f^i(x))=\alpha\}.$
	\end{abstract}

	%%%%%%%%%%%%%%%%%%%%%%%%%%%%%%%%%%%%%%%%%%%%
	
	\title[The irregular set has full metric mean dimension]{The irregular set for maps with almost weak specification property has full metric mean dimension}
	\maketitle

	%\baselineskip 15pt   % between lines
	%12pt-standard, 24pt-double space, 0.1in-narrow
	\parskip 10pt         % between paragraphs

	%%**********************************************************************
	%%%%%%%%%%%%%%%%%%%%%%%%%%%%%%%%%%%%%%%%%%%%%%%%%%%%%%%%%%%%%%%%%%%% SS1
	\section{Introduction}
	\subsection{Motivation}
	Let $(X,d,f)$ be a topological dynamical system (abbr. TDS), i.e. a compact metric space $(X,d)$ and a continuous transformation $f:X\to X$. For any continuous observable $\varphi:X\to \mathbb{R}$, the space $X$ has a natural multifractal decomposition
	\begin{equation*}
		X=\cup_{\alpha\in\mathbb{R}}K_{\varphi,\alpha}\cup I_\varphi,
	\end{equation*}where
	\begin{equation}\label{eq irregular set}
		K_{\varphi,\alpha}=\{x\in X:\lim_{n\to \infty}\frac{1}{n}\sum_{i=0}^{n-1}\varphi (f^i(x))=\alpha\}\mbox{ and } I_{\varphi}=\{x\in X:\lim_{n\to \infty}\frac{1}{n}\sum_{i=0}^{n-1}\varphi (f^i(x))\mbox{ DNE }\}.
	\end{equation}
	There are extensive literature studying the Bowen topological entropy and topological pressure of these decomposition sets. Here the Bowen topological entropy is defined by Bowen \cite{Bowen1973Topentropy} and topological pressure is defined by Pesin and Pitskel \cite{PesinPitskel1984toppressure} on noncompact sets, both of which can be equivalently defined by using Carathéodory dimension structure \cite{Pesin}. The Bowen topological entropy and pressure of  $K_{\varphi,\alpha}$ and $I_{\varphi}$ can be estimated by constructing some Moran-like fractals. In particular, a variational principle between the Bowen topological entropy of $K_{\varphi,\alpha}$ and measure-theoretic entropy has been established for systems with the specification-like property \cite{TV,PfisterSullivan2007galmost,TianVarandas2017}. Meanwhile, the irregular set $I_\varphi$ is either empty or carrying full Bowen topological entropy in systems with the specification property \cite{Shulin2005} (see \cite{Thompson2010irregularspecification} for the case of topological pressure), in systems with the almost specification property\cite{Thompson2012almostspecification} and in systems with the shadowing property\cite{Tian2018shadowing}.
	
	Note that on a compact smooth manifold with dimension greater than one, continuous transformations with infinite topological entropy  are generic in the space of all continuous transformations with uniform metric \cite{Koichi1980}. Recently, authors \cite{BobokTroubetzkoy2020} showed that in the space of continuous non-invertible maps of the unit interval preserving the Lebesgue measure, which is equipped with the uniform metric, the functions with the specification property and infinite topological entropy form a dense $G_\delta$ set. Since there are many  systems with the infinite entropy and the specification-like property, a more subtle question arises naturally:
	\begin{question}
		Given a system  having both the certain specification-like property and infinite topological entropy, and a continuous function $\varphi:X\to\mathbb{R}$,
		\begin{enumerate}
			\item for $\alpha\in\{\beta\in\mathbb{R}:\ K_{\varphi,\beta}\not=\emptyset\}$,	does $K_{\varphi,\alpha}$ carry more information besides the variational principle we mentioned before?
			\item    does $I_\varphi$ carry more information besides infinite Bowen topological entropy?
		\end{enumerate}
	\end{question}
	%In order to answer this question, we need to some new topological invariants to describe the systems with infinite entropy. There are many relative invariants to
	
	In order to quantify the complexity of systems with infinite entropy, Lindenstrauss and Weiss \cite{LindenstraussWeiss2000} introduced the (upper and lower) metric mean dimension \cite{LindenstraussWeiss2000}, which is a metric version of the mean dimension introduced by Gromov \cite{Gromov1999}. Roughly speaking, while $h_{\operatorname{top}}(f,\epsilon)$ measures the exponential growth rate of the number of $\epsilon$-distinguishable orbits as time advances and the topological entropy is the limit of $h_{\operatorname{top}}(f,\epsilon)$ as $\epsilon$ vanishes, the metric mean dimension measures the growth rate of $h_{\operatorname{top}}(f,\epsilon)$ as $\epsilon$ vanishes.
	Similar as the topological entropy, the metric mean dimension has a strong connection with ergodic theory, and lots of variational principles have been established, see \cite{S,LindenstraussTsukamoto2019,GS,LindenstraussTsukamoto2018,Wang2021} and reference therein.
	%Besides, mean dimension plays an key role in solving embedding problems, see \cite{LindenstraussWeiss2000,Gutman2017,GutmanTsukamoto2020,Lindenstrauss1999}.

	While $h_{\operatorname{top}}(f,\epsilon)$ in a way resembles the box dimension, the $\epsilon$-Bowen topological entropy in a way resembles the Hausdorff dimension. The Bowen (upper and lower) metric mean dimension can be defined similarly to capture the growth rate of the $\epsilon$-Bowen topological entropy as $\epsilon$ vanishes. When a set is compact and invariant, the Bowen metric mean dimension and the classical metric mean dimension of this set coincide. More recently, the author in \cite{Wang2021} studied the upper Bowen metric mean dimension of any compact subset and established a relation with the measure-theoretic lower entropy introduced in \cite{FengHuang2012}.
	
	Let us go back to our questions. Recently, for a TDS with the specification property and infinite topological entropy, Backes and Rodrigues \cite{LF} established a variational principle between the Bowen upper metric mean dimension of $K_{\varphi,\alpha}$ and growth rates of measure-theoretic entropy of partitions decreasing in diameter associated to those measure $\mu$ such that $\int\varphi d\mu=\alpha$. In this paper, we consider systems satisfying the almost weak specification property, which is weaker than the specification property.
	
	%Naturally, we consider that for a weaker specification-like property, whether the corresponding result is true?
	%In this paper, we use the almost weak specification property \cite{AnthonyTerry2016} to replace the specification property. Namely,
	\begin{definition}\label{def almostweakspecificationproperty}
		A system $(X,d,f)$ is said to satisfy the almost weak specification property if for any $\epsilon>0$, there exists a tempered function  $L_\epsilon:\mathbb{N}\to \mathbb{N}$ (i.e., nondecreasing and $\lim_{n\to\infty}{L_{\epsilon}(n)}/{n}=0$) such that for any sequence of points $x_1,\ldots,x_k\in X$ and positive integers $a_1\le b_1<a_2\le b_2<\ldots<a_k\le b_k$ satisfying $a_{j+1}-b_j\ge L_{\epsilon}(b_{j+1}-a_{j+1})$ for all  $j\in \{1,...,k-1\}$, there exists $y\in X$ such that
		\begin{equation*}
			d(f^i(y),f^{i-a_j}(x_j))< \epsilon,\mbox{ for any }i\in[a_j,b_j] \text{ and }j\in \{1,...,k\}.
		\end{equation*}
	\end{definition}The notion of the almost weak specification property was first introduced in \cite{Marcus1980} without a name. It was called the almost weak specification in some references such as \cite{AnthonyTerry2016,MR1041229,Sun2017}, called the weak specification in \cite{MR3546668}, and suggested to be called the tempered specification in \cite{sun2019ergodic}. Marcus \cite{Marcus1980} proved that quasi-hyperbolic toral automorphisms satisfy the almost weak specification. In \cite{MR1041229}, Dateyama showed every automorphism of a compact metric abelian group is ergodic under the Haar measure if and only if it satisfies the
	almost weak specification. More relevant results can be seen in \cite{MR3546668}. Relations between specification-like properties can be found in \cite{SunPeng,sun2019ergodic}.
	
	In this paper, we focus on the second question and show that for the TDS with the almost weak specification property and infinite topological entropy, either $I_\varphi$ is empty or its Bowen (upper and lower) metric mean dimension coincides with the classical (upper and lower) metric mean dimension of $(X,d,f)$. We believe that the same method can be applied to show that the main result, a variational principle on level sets, in \cite{LF} also holds for such systems. To state our main results precisely, we start with some definitions.

	%In this paper, we firstly prove the reult in \cite{LF} also holds when the TDS with the almost weak specification property.
	%Next, we focus on the second question and show that for the TDS with the almost weak specification property and infinite entropt, either $I_\varphi$ is empty or its Bowen (upper and lower) metric mean dimension of $I_\varphi$ coincides the classical (upper and lower) metric mean dimension of $(X,d,f)$.
	
	%When we use the almost weak specification property, the gap depends on the length of the shadowing orbit. Then we cannot use the classical method to construct the Moran-like fractal to complete the proof, as this results in too much error. In this paper, we use some observation to reduce the number of times we use the almost weak specification property to avoid the problem above-mentioned.
	
	%The difficulty in proving the main results is twofold. Firstly,

	\subsection{Metric mean dimension}
	Let $(X,d,f)$ be a TDS. Given $n \in \mathbb{N}$, we define the Bowen metric $d_n$ on $X$ by
	$d_n(x, y)=\max_{0\le i\le n-1} d(f^{i}x,f^{i}y) .$
	Given $\epsilon>0, n \in \mathbb{N}$ and $x \in X$, we define the $(n,\epsilon)$-ball around $ x$ by $
	B_n(x, \epsilon)=\left\{y \in X : d_n(x, y)<\epsilon\right\}.$
	
	By employing the Carathéodory dimension structure (see \cite{Pesin}), the Bowen metric mean dimension is defined as follows.
	Given a nonempty set $Z \subset X$, let
	\begin{equation}\label{def m Z,s,Nepsilon}
		m(Z, s, N, \epsilon)=\inf _{\Gamma}\left\{\sum_{i \in I} \exp \left(-s n_i\right)\right\},
	\end{equation}
	where the infimum is taken over all finite or countable collection $\Gamma=\left\{B_{n_i}\left(x_i, \epsilon\right)\right\}_{i \in I}$ with $Z\subset\cup_{i\in I}B_{n_i}\left(x_i, \epsilon\right)$ and $\min\{n_i:i\in I\} \geq N$. Note that $m(Z,s,N,\epsilon)$ does not decrease as $N$ increases, and therefore the following limit exists
	$$
	m(Z, s, \epsilon)=\lim _{N \rightarrow \infty} m(Z, s, N, \epsilon) .
	$$
	There exists a critical number $h_{top}^B(Z,f,\epsilon)\in\mathbb{R}$ such that
	\begin{equation*}
		m(Z,s,\epsilon)=\begin{cases}
			+\infty, & \mbox{if } s<h_{top}^B(Z,f,\epsilon) \\
			0, & \mbox{if }s>h_{top}^B(Z,f,\epsilon).
		\end{cases}
	\end{equation*}Note that $m(Z,h_{top}^B(Z,f,\epsilon),\epsilon)$ could be $+\infty$, $0$ or some positive finite number. The Bowen topological entropy is defined by $h_{top}^B(Z,f)=\lim_{\epsilon\to 0}h_{top}^B(Z,f,\epsilon)$ (see \cite{Pesin} or \cite{TV}).
	The Bowen upper and lower metric mean dimension of $f$ on $Z$ with respect to $d$ are defined by
	\begin{equation}\label{2}
		\overline{\operatorname{mdim}}_\mathrm{M}^B(Z, f, d)=\limsup _{\epsilon \rightarrow 0} \frac{h_{top}^B(Z, f, \epsilon)}{|\log \epsilon|}\mbox{ and }\underline{\operatorname{mdim}}_\mathrm{M}^B(Z, f, d)=\liminf _{\epsilon \rightarrow 0} \frac{h_{top}^B(Z, f, \epsilon)}{|\log \epsilon|},
	\end{equation}respectively.
	By the definition, if $Z_1\subset Z_2$ are nonempty, then
	\begin{equation}\label{eq subset leq}
		\overline{\operatorname{mdim}}_\mathrm{M}^B(Z_1, f, d)\leq \overline{\operatorname{mdim}}_\mathrm{M}^B(Z_2, f, d)\text{ and } \underline{\operatorname{mdim}}_\mathrm{M}^B(Z_1, f, d)\leq \underline{\operatorname{mdim}}_\mathrm{M}^B(Z_2, f, d).
	\end{equation}
	
	The classical metric mean dimension is defined as follows.
	Given $n \in \mathbb{N}$ and $\epsilon>0$,  we say that a set $E \subset X$ is $(n, \epsilon)$-separated if $d_n(x, y)>\epsilon$ for every $x\not=y \in E$. Denote by $s(f,X, n, \epsilon)$ the largest cardinality of all maximal $(n, \epsilon)$-separated subsets of $X$, which is finite due to the compactness of $X$.
	The classical upper and lower metric mean dimension of $f$ with respect to $d$ are given by
	\begin{equation}\label{1}
		\overline{\operatorname{mdim}}_{\mathrm{M}}(X, f, d)=\limsup _{\epsilon \rightarrow 0} \frac{h_{top}(f, X,\epsilon)}{|\log \epsilon|}\text{ and }\underline{\operatorname{mdim}}_{\mathrm{M}}(X, f, d)=\liminf _{\epsilon \rightarrow 0} \frac{h_{top}(f, X,\epsilon)}{|\log \epsilon|},
	\end{equation}respectively,
	where
	\begin{equation*}
		h_{top}(f, X,\epsilon)=\limsup_{n \rightarrow \infty} \frac{1}{n} \log s(f, X,n, \epsilon)=\lim_{n \rightarrow \infty} \frac{1}{n} \log s(f, X,n, \epsilon).
	\end{equation*}
	Recall that the classical topological entropy of $f$ is given by
	$
	h_{top}(f,X)=\lim _{\epsilon \rightarrow 0} h_{top}(f, X,\epsilon).
	$
	It is clear that the metric mean dimension vanishes if $h_{top}(f,X)<\infty$.
	
	The following proposition is stated in several papers\cite{LF,Wang2021}, but we were not able to locate a clear reference to it. Therefore, we give a  proof in appendix for the sake of completeness.
	\begin{proposition}\label{propositiob coincides}
		For any $f$-invariant and compact nonempty subset $Z\subset X$, one has
		\begin{equation*}
			\overline{\operatorname{mdim}}_\mathrm{M}^B(Z, f, d)=\overline{\operatorname{mdim}}_{\mathrm{M}}(Z, f, d)\text{ and }\underline{\operatorname{mdim}}_\mathrm{M}^B(Z, f, d)=\underline{\operatorname{mdim}}_{\mathrm{M}}(Z, f, d).
		\end{equation*}
	\end{proposition}
	\subsection{Main result} The following is our main result.
	\begin{theorem}\label{main}
		Let $(X, d,f)$ be a TDS satisfying the almost weak specification property, and let $I_\varphi$ be the irregular set of $\varphi$ defined in \eqref{eq irregular set}. If  $\overline{\operatorname{mdim}}_{\mathrm{M}}(X, f, d)<\infty$ (resp. $\underline{\operatorname{mdim}}_{\mathrm{M}}(X, f, d)<\infty$), then either $I_\varphi=\emptyset$ or
		\begin{equation}\label{eq Bowen=mean}
			\overline{\operatorname{mdim}}_{\mathrm{M}}^B(I_\varphi, f, d)=\overline{\operatorname{mdim}}_{\mathrm{M}}(X, f, d) \quad( \text{resp. }\underline{\operatorname{mdim}}_{\mathrm{M}}^B(I_\varphi, f, d)=\underline{\operatorname{mdim}}_{\mathrm{M}}(X, f, d)).
		\end{equation}
	\end{theorem}

	The paper is organized as follows. In Sec. \ref{section preliminary}, we introduce several preliminary lemmas. In Sec. \ref{section proof of Theorem}, we prove Theorem \ref{main}. In Sec. \ref{section 4}, we discuss some applications and examples. The proofs of some propositions and lemmas are addressed in Appendix \ref{appendix} for the sake of completeness.
	
	\section{Preliminary Lemmas}\label{section preliminary}
	In this section, we give several preliminary lemmas for the proof of our main result. These lemmas do not rely on any specification-like property.

	\subsection{Measure approximation}
	Denote $M(X)$ to be the space of all Borel probability measures on $X$, $M(X,f)$ to be the space of all $f$-invariant Borel probability measures on $X$ and denote $M^e(X,f)$ to be the space of all $f$-ergodic invariant Borel probability measures on $X$.
	
	Denote $\mathcal{P}_X$ to be the set of all finite Borel measurable partitions of $X$.
	For any  $\xi\in\mathcal{P}_X$, denote $\partial \xi$ to be the boundary of the partition,  which is the union of the boundaries of all the elements of the partition. For any finite open cover $\mathcal{U}$ of $X$, notation $\xi \succ \mathcal{U}$ means that $\xi\in \mathcal{P}_X$ and refines $\mathcal{U}$, that is, each element of $\xi$ is contained in an element of $\mathcal{U}$.
	
	The following measure approximation lemma is an intermediate step of \cite[Lemma 3.3]{LF}. For the sake of completeness, we give a proof in appendix.
	\begin{lemma}\label{lem5}
		Given a finite open cover $\mathcal{U}$ of $X$, $\varphi\in C(X,\mathbb{R})$ and a measure $\mu\in M(X,f)$. For any $\delta>0$, there exists a measure $\nu\in M(X,f)$ satisfying
		\begin{itemize}
			\item[(1)] $\nu=\sum_{i=1}^{j} \lambda_i \nu_i$, where $\lambda_i>0, \sum_{i=1}^{j} \lambda_i=1$ and $\nu_i \in  M^e(X,f)$;
			\item[(2)] $\inf _{\xi\succ \mathcal{U}} h_\mu(f, \xi)\leq\sum_{i=1}^j\lambda_i\inf _{\xi\succ \mathcal{U}} h_{\nu_i}(f, \xi)+\delta;$
			\item[(3)] $\left|\int_X \varphi d \nu-\int_X \varphi d \mu\right|<\delta$.
		\end{itemize}
	\end{lemma}

	\subsection{Lebesgue number of open covers}
	Recall that the Lebesgue number of an open cover $\mathcal{U}$ of $X$, denoted by $\operatorname{Leb}(\mathcal{U})$, is the largest number $\epsilon>0$ with the property that every open ball of radius $\epsilon$ is contained in an element of $\mathcal{U}$. Denote $\operatorname{diam} (\mathcal{U})=\max\{\operatorname{diam} (U_i):\ U_i\in\mathcal{U}\}$.  The following lemma can be found in \cite[Lemma 3.4]{GS}.
	
	\begin{lemma}\label{lem4}
		For every $\epsilon>0$, there exists a finite open cover $\mathcal{U}$ of $X$ such that $\operatorname{diam}(\mathcal{U}) \leq \epsilon$ and $\operatorname{Leb}(\mathcal{U}) \geq \frac{\epsilon}{4}$.
	\end{lemma}
	Given a measure $\mu\in M(X,f)$, for $\delta \in(0,1), n \in \mathbb{N}$ and $\epsilon>0$, denote $\widetilde{N}_\mu^\delta(n, \epsilon)$ to be the smallest number of sets with diameter at most $\epsilon$ in the metric $d_n$, whose union has $\mu$-measure larger than $1-\delta$. Denote $N_\mu^\delta(n, \epsilon)$  to be the smallest number of any $(n, \epsilon)$-balls, whose union has $\mu$-measure larger than $1-\delta$. It is clear that
	\begin{equation}\label{3}
		\widetilde{N}_\mu^\delta(n, 2 \epsilon) \leq N_\mu^\delta(n, \epsilon) \leq \widetilde{N}_\mu^\delta(n, \epsilon) .
	\end{equation}
	Moreover, we have the following inequality.
	\begin{lemma}\cite[Lemma 8]{S}\label{lem6}
		Let $\mu\in M^e(X,f)$ and $\mathcal{U}$ be a finite open cover of $X$ with $\operatorname{diam}(\mathcal{U}) \leq \epsilon_1$ and $\operatorname{Leb}(\mathcal{U}) \geq \epsilon_2$. Then  for $\delta \in(0,1)$,
		$$
		\widetilde{N}_\mu^\delta\left(n, \epsilon_1\right) \leq \mathcal{N}^ \delta_\mu\left(\mathcal{U}^n\right) \leq N_\mu^\delta\left(n, \epsilon_2\right),
		$$
		where $\mathcal{N}^ \delta_\mu\left(\mathcal{U}^n\right)$ is the smallest number of elements of $\mathcal{U}^n:=\bigvee_{j=0}^{n-1}f^{-j}\mathcal{U}$ needed to cover a subset of $X$ whose $\mu$-measure is at least $1-\delta$.
	\end{lemma}

	\subsection{Measure-theoretic entropy}
	Given $\mu \in M(X,f)$ and a finite Borel measurable partition $\xi=\{C_1,\ldots,C_k\}$$\in\mathcal{P}_X$, let $ H_\mu(\xi)=-\sum_{i=1}^k \mu\left(C_i\right) \log \mu\left(C_i\right).$
	Then the Measure-theoretic entropy of $\mu$ is given by
	\begin{equation*}
		h_\mu(f)=\sup_{\xi\in\mathcal{P}_X} h_\mu(f, \xi),
	\end{equation*}where $h_\mu(f, \xi)=\lim _{n \to\infty} \frac{1}{n} H_\mu(\bigvee_{j=0}^{n-1} f^{-j}\xi)$.
	The following is the Katok entropy formula.
	\begin{lemma}\cite[Theorem I.I]{K}\label{lem8}
		Let $\mu \in M^e(X,f)$. Then for any $\delta\in (0,1)$,
		\begin{align*}
			h_{\mu}(f) & =\lim_{\epsilon \rightarrow 0}\underline{h}_\mu(f,\epsilon,\delta)=\lim_{\epsilon \rightarrow 0}\overline{h}_\mu(f,\epsilon,\delta),
		\end{align*}where
		\begin{equation}\label{def underbar h}
			\underline{h}_\mu(f,\epsilon,\delta)=\liminf_{n\to\infty}\frac{\log N^\delta_{\mu}(n,\epsilon)}{n}\mbox{ and }
			\overline{h}_\mu(f,\epsilon,\delta)=\limsup_{n\to\infty}\frac{\log N^\delta_{\mu}(n,\epsilon)}{n}.
		\end{equation}
	\end{lemma}
	
	\begin{lemma}\cite[Theorem 4.4]{Shapira2007}\label{lem7}
		Let $\mathcal{U}$ be a finite open cover of $X$ and $\mu \in M^e(X,f)$. For any $\delta\in (0,1)$, one has
		$$
		\lim_{n \to \infty}\frac{1}{n}\log\mathcal{N}^\delta_\mu\left(\mathcal{U}^n\right)=\inf _{\xi\succ \mathcal{U}} h_\mu(f,\xi).
		$$
	\end{lemma}
	\begin{lemma}\label{lemma leq leq}
		Let $\mu\in M^e(X,f)$ and $\mathcal{U}$ be a finite open cover of $X$ with $\operatorname{diam}(\mathcal{U}) \leq \epsilon_1$ and $\operatorname{Leb}(\mathcal{U}) \geq \epsilon_2$. Then for any $\delta\in(0,1)$, one has
		\begin{equation}\label{eq lemma2.6}
			\underbar{h}_\mu(f,\epsilon_1,\delta)\leq \inf_{\xi\succ \mathcal{U}}h_{\mu}(f,\xi)\leq \underline{h}_{\mu}(f,\epsilon_2,\delta).
		\end{equation}
	\end{lemma}
	\begin{proof}
		By Lemma \ref{lem6} and \eqref{3}, one has
		\[N_{\mu}^\delta\left(n, \epsilon_1\right) \leq \mathcal{N}^ \delta_{\mu}\left(\mathcal{U}^n\right) \leq N_{\mu}^\delta\left(n, \epsilon_2\right).\]
		Taking Lemma \ref{lem8} and Lemma \ref{lem7} into account, one immediately obtains \eqref{eq lemma2.6}.
	\end{proof}
	
	\subsection{Variational principles}
	The following variational principle for the metric mean dimension  was obtained by Gutman and \'{S}piewak \cite{GS}.
	\begin{lemma}\cite[Theorem 3.1]{GS}\label{lemma 2.7}
		Let $(X,d,f)$ be a TDS. Then
		\begin{equation*}
			\overline{\operatorname{mdim}}_{\mathrm{M}}(X, f, d)=\limsup _{\epsilon \rightarrow 0} \frac{1}{|\log \epsilon|} \sup _{\mu \in M(X,f)} \inf _{|\xi|<\epsilon} h_\mu(f, \xi)
		\end{equation*}and
		\begin{equation*}
			\underline{\operatorname{mdim}}_{\mathrm{M}}(X, f, d)=\liminf _{\epsilon \rightarrow 0} \frac{1}{|\log \epsilon|} \sup _{\mu \in M(X,f)} \inf _{|\xi|<\epsilon} h_\mu(f, \xi) ,
		\end{equation*} where $|\xi|$ denotes the diameter of the partition $\xi\in\mathcal{P}_X$ and the infimum is taken over $\xi\in\mathcal{P}_X$ satisfying  $|\xi|<\epsilon$.
	\end{lemma}
	Recently, Shi \cite{S} obtained the following variational principle, which can be seen as the Katok entropy formula for the metric mean dimension.
	\begin{lemma}\cite[Theorem 9]{S}\label{lem9}
		Let $(X,d,f)$ be a TDS. For every $\delta \in(0,1)$, one has
		\begin{equation}\label{eq lemma 2.8}
			\overline{\operatorname{mdim}}_\mathrm{M}(X, d, f)=\limsup _{\epsilon \rightarrow 0} \frac{1}{|\log {\epsilon}|} \sup _{\mu \in M^e(X,f)} \overline{h}_\mu(f,\epsilon,\delta),
		\end{equation}
		and
		\begin{equation}\label{eq 2.8}
			\underline{\operatorname{mdim}}_\mathrm{M}(X, d, f)=\liminf _{\epsilon \rightarrow 0} \frac{1}{|\log {\epsilon}|} \sup _{\mu \in M^e(X,f)} \overline{h}_\mu(f,\epsilon,\delta).
		\end{equation}
	\end{lemma}

	\section{Proof of Theorem \ref{main}}\label{section proof of Theorem}
	In this section, we prove Theorem \ref{main}. We assume $I_\varphi \not=\emptyset$ and show \eqref{eq Bowen=mean}. We first consider the case of the upper metric mean dimension. Note that $I_\varphi\subset X$, and therefore
	\begin{equation*}
		\overline{\operatorname{mdim}}_{\mathrm{M}}^B(I_\varphi, f, d)\overset{\eqref{eq subset leq}}\leq \overline{\operatorname{mdim}}_{\mathrm{M}}^B(X, f, d).
	\end{equation*}Proposition \ref{propositiob coincides} implies that
	\begin{equation*}
		\overline{\operatorname{mdim}}_{\mathrm{M}}^B(I_\varphi, f, d)\leq \overline{\operatorname{mdim}}_{\mathrm{M}}(X, f, d).
	\end{equation*} In the following, we need to show \begin{equation*}
		\overline{\operatorname{mdim}}_{\mathrm{M}}^B(I_\varphi, f, d)\geq \overline{\operatorname{mdim}}_{\mathrm{M}}(X, f, d).
	\end{equation*}Denote $S=\overline{\operatorname{mdim}}_{\mathrm{M}}(X, f, d)<\infty$. We only need to consider the case that $S>0$. In the following, we will show that for any $\gamma\in \left(0,\min\left\{S/{7},1\right\}\right)$ small enough,
	$
	\overline{\operatorname{mdim}}_{\mathrm{M}}^B(I_\varphi, f, d)\geq S-6\gamma.
	$ Let us begin with the following lemma.
	\begin{lemma}\label{lem2}
		For any sufficiently small $\gamma\in \left(0,\min\left\{S/{7},1\right\}\right)$, there exists $\epsilon_0=\epsilon_0(\gamma)>0$  such that
		\begin{align}
			&|\log 5\epsilon_0|>1,\label{eq 3.15}\\
			& S-\gamma/2\le\frac{1}{|\log 5\epsilon_0|} \sup _{\mu \in M(X,f)} \inf _{|\xi|<5\epsilon_0} h_\mu(f, \xi),\label{eq 3.11}\\
			&\frac{h_{\operatorname{top}}^B(I_{\varphi},f,\epsilon_0/4)}{|\log{\epsilon_0/4}|}\le \sup_{\epsilon\in(0,5\epsilon_0)}\frac{h_{\operatorname{top}}^B(I_{\varphi},f,\epsilon)}{|\log{\epsilon}|}\le\overline{\operatorname{mdim}}_\mathrm{M}^B(I_{\varphi}, f, d)+\gamma,\label{eq 3.12}\\
			&(\overline{\operatorname{mdim}}_\mathrm{M}^B(I_{\varphi}, f, d)+\gamma )\cdot\frac{|\log\epsilon_0/4|}{|\log5\epsilon_0|}
			\le\overline{\operatorname{mdim}}_\mathrm{M}^B(I_{\varphi}, f, d)+2\gamma.\label{eq 3.13}
		\end{align}
		Moreover, there exist $\mu_1, \mu_2 \in M\left(X,f\right)$ such that
		\begin{equation}\label{eq 3.14}
			\int \varphi d \mu_1 \neq \int \varphi d \mu_2\mbox{ and }\frac{1}{|\log 5\epsilon_0|} \inf _{|\xi|<5\epsilon_0} h_{\mu_i}(f, \xi)>S-\gamma\mbox{ for $i=1,2$.}
		\end{equation}
	\end{lemma}
	\begin{proof}[Proof of Lemma \ref{lem2}]
		We first pick $\epsilon^\prime>0$ sufficiently small such that \eqref{eq 3.15} and \eqref{eq 3.13} can be achieved for any $\epsilon_0\in (0,\epsilon^\prime)$.
		According to  Lemma \ref{lemma 2.7}, \eqref{2} and \eqref{eq lemma 2.8}, we can pick $\epsilon_0=\epsilon_0(\gamma)\in (0,\epsilon^\prime)$ satisfying
		\begin{align*}
			& S-\gamma/2\le\frac{1}{|\log 5\epsilon_0|} \sup _{\mu \in M(X,f)} \inf _{|\xi|<5\epsilon_0} h_\mu(f, \xi)\\
			&\sup_{\epsilon\in(0,5\epsilon_0)}\frac{h_{\operatorname{top}}^B(I_{\varphi},f,\epsilon)}{|\log{\epsilon}|}\le\overline{\operatorname{mdim}}_\mathrm{M}^B(I_{\varphi}, f, d)+\gamma
		\end{align*}and
		\begin{equation*}
			\sup_{\epsilon\in (0,5\epsilon_0)}\frac{1}{|\log\epsilon|}\sup_{\mu\in M^e(X,f)}\bar{h}_\mu(f,\epsilon,\gamma)<S+\gamma,
		\end{equation*}respectively. Then \eqref{eq    3.11} and \eqref{eq 3.12} are direct corollaries of the above three inequalities.
		
		Now we show \eqref{eq 3.14}.	
		By \eqref{eq     3.11}, we can choose $\mu_1\in M(X,f)$ such that $$\frac{1}{|\log 5\epsilon_0|} \inf _{|\xi|<5\epsilon_0} h_{\mu_1}(f, \xi)>S-2\gamma/3.$$ By the assumption that $I_\varphi\not=\emptyset$, we claim that there exists  $\nu \in M(X,f)$ satisfying
		\begin{equation}\label{claim <}
			\int \varphi d \mu_1 \neq \int \varphi d \nu.
		\end{equation}Let $\mu_2=t \mu_1+(1-t) \nu$, where $t \in(0,1)$ is chosen sufficiently close to $1$ so that $$\frac{1}{|\log 5\epsilon_0|} \inf _{|\xi|<5\epsilon_0} h_{\mu_2}(f, \xi)>S-\gamma.$$
		Therefore, the proof of Lemma \ref{lem2} is completed if we prove the claim \eqref{claim <}. In fact, by the definition of $I_\varphi\not=\emptyset$, there exists a point $x\in X$ such that $\lim_{n\to\infty}\frac{1}{n}\sum_{i=0}^{n-1}\varphi (f^i(x))$ has a subsequence converging to a number other than $\int\varphi d\mu_1$, named $\lim_{k\to\infty}\frac{1}{n_k}\sum_{i=0}^{n_k-1}\varphi(f^i(x))=C\not=\int\varphi d\mu_1$. We define $\nu_k:=\frac{1}{n_k}\sum_{i=1}^{n_k-1}\delta_{f^i(x)}\in M(X)$, and let $\nu$ be any weak$^*$ limit measure of $\nu_k$. Then $\nu\in M(X,f)$ is the desired measure.
	\end{proof}
	In the following, we fix $\gamma$, $\epsilon_0>0$ and the measures $\mu_1,\mu_2$ obtained in Lemma \ref{lem2}.
	For $i=1,2 $, we denote
	$\alpha_i=\int \varphi d \mu_i$ . In the next section, we shall prove that the following deviation sets
	\begin{equation*}
		P(\alpha_i,Error,n):\overset{def}{=}\left\{x\in X:\ \left|\frac{1}{n}\sum_{j=0}^{n-1}\varphi(f^j(x))-\alpha_i\right|<Error\right\}\mbox{ for } i=1,2
	\end{equation*}are nonempty sets for sufficiently large $n\in\mathbb{N}$. In the following, Sec. \ref{subsection 3.1}-\ref{subsection 3.3}, we are going to construct a Moran-like fractal contained in $I_\varphi$ by gluing points in the above deviation set alternatively. Moreover, this Moran-like fractal has to satisfy a certain entropy distribution principle.
	
	\begin{remark}
		One of difficulties in the proof is that measures $\mu_1$ and $\mu_2$ chosen in Lemma \ref{lem2} are only $f$-invariant, and they  may not be ergodic. While in the proof for the case that irregular sets carry full topological pressure for systems with specification-like property (see for example \cite{Thompson2012almostspecification}), $\mu_1$ and $\mu_2$ can be chosen to be ergodic by using the fact that ergodic measures are entropy-dense for systems with approximate product property \cite[Theorem 2.1]{PfisterSullivan2005}. %and we note that the approximate product property is weaker than almost specification and orbit gluing property. We also remark that Lima and Varandas's proof did not apply this fact, but this fact can simplify their proof.
		In this paper, we note that $\frac{1}{5\epsilon_0}\inf_{|\xi|<5\epsilon_0}h_\mu(f,\xi)$ for any measure $\mu\in M(X,f)$ may not be approachable by $\{\frac{1}{5\epsilon_0}\inf_{|\xi|<5\epsilon_0}h_\mu(f,\xi):\mu\in M^e(X,f)\}$.
	\end{remark}

	\subsection{Construction of the Moran-like fractal}\label{subsection 3.1}
	The idea of construction was inspired by \cite{Thompson2010irregularspecification} and \cite{Thompson2012almostspecification}. The absence of ergodicity of $\mu_1$ and $\mu_2$ can be bypassed by using Lemma \ref{lem5}. The details are contained in the following technical lemma.
	\begin{lemma}\label{lem3}
		For any $\delta\in (0,\gamma/2)$, $N'\in\mathbb{N}$ and $i=1,2$, there exists an integer $\hat{n}_i:=\hat{n}_i(\delta,N')$ such that $\hat{n}_i>N'$ and
		\begin{itemize}
			\item [(1)]$P(\alpha_i,4\delta,\hat{n}_i):=\{x\in X:|\frac{1}{\hat{n}_i}\sum_{j=0}^{\hat{n}_i-1}\varphi(f^j(x))-\alpha_i|<4\delta\}\neq\emptyset$;
			\item[(2)]$M(\alpha_i,4\delta,\hat{n}_i,\frac{9\epsilon_0}{8})$, the largest cardinality of maximal $(\hat{n}_i,\frac{9\epsilon_0}{8})$-separated set  in $P(\alpha_i,4\delta,\hat{n}_i)$, satisfies
			\[M(\alpha_i,4\delta,\hat{n}_i,9\epsilon_0/8)\ge \exp\left\{\hat{n}_i\left(\inf _{|\xi|<5\epsilon_0} h_{\mu_i}(f, \xi)-2\gamma\right)\right\}.\]
		\end{itemize}
	\end{lemma}
	\begin{proof}[Proof of Lemma \ref{lem3}]
		We only prove the case for $i=1$, as the other case can be proved similarly. Fix any $\delta\in (0,\gamma/2)$ and $N'\in\mathbb{N}$. By Lemma \ref{lem4}, there exists a finite open cover $\mathcal{U}$ of $X$ such that
		\begin{equation}\label{Pick U}
			\operatorname{diam}(\mathcal{U})\le 5\epsilon_0 \quad\text{ and }\quad \operatorname{Leb}(\mathcal{U})\ge {5\epsilon_0}/{4}.
		\end{equation}
		By Lemma \ref{lem5}, there exists $\nu\in M(X,f)$ satisfying
		\begin{align}
			&\nu=\sum_{i=1}^{k} \lambda_i \nu_i,\text{ where }\lambda_i>0,\text{ }\sum_{i=1}^{k} \lambda_i=1\text{ and }\nu_i \in  M^e(X,f);\label{23}\\
			&\inf _{\xi\succ \mathcal{U}} h_{\mu_1}(f, \xi) \leq\sum_{i=1}^k\lambda_i\inf _{\xi\succ \mathcal{U}} h_{\nu_i}(f, \xi)+\delta;\label{24}\\
			&\left|\int_X \varphi d \nu-\int_X \varphi d \mu_1\right|<\delta.\label{25}
		\end{align}
		By the ergodicity of $\nu_i$, there exists $N_0\in\mathbb{N}$ such that
		\[Y_i(N_0)=\left\{x\in X:\left|\frac{1}{l}\sum_{j=0}^{l-1}\varphi(f^j(x))-\int \varphi d\nu_i\right|<\delta, \text{ for any }l\ge N_0\right\}\]
		has $\nu_i$-measure at least $1-\gamma$ for each $i\in\{1,2,\ldots,k\}$.
		Applying Lemma \ref{lemma leq leq} and \eqref{Pick U} on each $\nu_i\in M^e(X,f)$, we obtain
		\begin{equation}\label{eq leq leq infty}
			\underline{h}_{\nu_i}(f,5\epsilon_0,\gamma)\le \inf_{\xi\succ \mathcal{U}}h_{\nu_i}(f,\xi)\le\underline{h}_{\nu_i}(f,5\epsilon_0/4,\gamma)\leq \bar{h}_{\nu_i}(f,5\epsilon_0/4,\gamma)\leq h_{\mathrm{top}}(f,5\epsilon_0/4)<\infty.
		\end{equation}
		For any $n\in\mathbb{N}$, denote $C_i(n,5\epsilon_0/4)$ to be a maximal $(n,5\epsilon_0/4)$-separated set of $Y_i(N_0)$ with the largest cardinality. Note that $\# C_i(n,5\epsilon_0/4)$ is greater than the minimal number of $(n,5\epsilon_0/4)$-spanning set for $Y_i(N_0)$.  By the definition of $\underline{h}_{\nu_i}(f,5\epsilon_0/4,\gamma)$ in \eqref{def underbar h}, there exists $N_i\in\mathbb{N}$ such that for any $n>N_i$, we have
		\begin{equation}\label{4}
			\begin{split}
				\#C_i(n,5\epsilon_0/4)\ge & N_{\nu_i}^\gamma(n,5\epsilon_0/4)\ge \exp\left\{n\left(\underline{h}_{\nu_i}(f,5\epsilon_0/4,\gamma)-{\gamma}/{2}\right)\right\}\\
				\overset{\eqref{eq leq leq infty}}\ge&\exp\left\{n\left(\inf_{\xi\succ \mathcal{U}}h_{\nu_i}(f,\xi)-{\gamma}/{2}\right)\right\}.
			\end{split}
		\end{equation}
		
		We choose $\widetilde{N}>N'$ such that for any $n>\widetilde N$,
		\begin{equation}\label{18}
			n_i:= [\lambda_i n]>\max\{N_i,N_0\},\mbox{ for }i=1,...,k,
		\end{equation}where $[\lambda_in]$ is the largest integer less than or equal to $\lambda_i  n$.
		We choose $\epsilon\in(0,\epsilon_0)$  satisfying
		\begin{equation}\label{pick epsilon}
			|\varphi(x)-\varphi(y)|<\delta,\mbox{ whenever }d(x,y)<\epsilon.
		\end{equation}
		Let $L_{\epsilon/16}:\mathbb{N}\to\mathbb{N}$ be the tempered function in Definition \ref{def almostweakspecificationproperty} corresponding to $\epsilon/16$. We pick $N>\widetilde{N}$ such that for any $n>N$,
		\begin{equation}\label{pick n}
			k\leq \frac{\delta n}{4\|\varphi\|_{C^0}+1}\mbox{ and }\frac{2\|\varphi\|_{C^0}\cdot \sum_{i=2}^{k}L_{\epsilon/16}(n_i)}{n}<\frac{\delta}{2},
		\end{equation} and moreover,
		\begin{equation}\label{pick n 2 Lepsilon}
			\frac{n_i}{n_i+L_{\epsilon/16}(n_{i})}\geq \frac{\inf_{\xi\succ \mathcal{U}}h_{\mu_1}(f,\xi)-2\gamma}{\inf_{\xi\succ \mathcal{U}}h_{\mu_1}(f,\xi)-\frac{3}{2}\gamma}\mbox{ for }i=2,...,k.
		\end{equation}
		We can achieve \eqref{pick n} and \eqref{pick n 2 Lepsilon} since $k$ is fixed and $L_{\epsilon/16}$ is tempered.
		
		Now we pick any $n>N$. By \eqref{4} and \eqref{18}, one has, for each $i=1,2,\ldots,k$,
		\begin{equation}\label{5}
			\#C_i(n_i,5\epsilon_0/4)\ge\exp\left\{n_i\left(\inf_{\xi\succ \mathcal{U}}h_{\nu_i}(f,\xi)-\gamma/2\right)\right\}.
		\end{equation}
		For any $k$-tuple $(x_1,\ldots,x_k)$ with $x_i\in C_i(n_i,5\epsilon_0/4)$ for $i=1,\ldots,k$, by the almost weak specification property (Definition \ref{def almostweakspecificationproperty}), there exists $y:=y(x_1,...,x_k)\in X$  such that
		\begin{equation}\label{s1}
			d_{n_j}(f^{\sum_{i=0}^{j-1}(n_i+L_{\epsilon/16}(n_{i+1}))}y,x_j)<\epsilon/16 \mbox{ for }j=1,...,k,
		\end{equation}
		where $\{L_{\epsilon/16}(n_{i+1})\}_{i=1}^{k-1}$ are gaps.
		We collect all shadowing points into the set
		\begin{equation*}
			E=\{y(x_1,...,x_k):\ (x_1,...,x_k)\in\prod_{i=1}^{k}C_i(n_i,5\epsilon_0/4)\}.
		\end{equation*}
		Denote $\hat{n}_1=\sum_{i=1}^{k}n_i+\sum_{i=1}^{k-1}L_{\epsilon/16}(n_{i+1}).$ Note that for any two different $k$-tuples $(x_1,...,x_k)$ and $(z_1,...,z_k)$\ with $x_i,z_i\in C_i(n_i,5\epsilon_0/4)$ for $i=1,...,k$, then $x_j\not=z_j$ for some $j\in\{1,...,k\}$. Denote $y_1=y(x_1,...,x_k)$ and $y_2=y(z_1,...,z_k)$, then we have
		\begin{equation}\label{7}
			\begin{split}
				d_{\hat{n}_1}(y_1,y_2)
				\geq& d_{n_j}(x_j,z_j)-d_{n_j}(f^{\sum_{i=0}^{j-1}(n_i+L_{\epsilon/16}(n_{i+1}))}y_1,x_j)-d_{n_j}(f^{\sum_{i=0}^{j-1}(n_i+L_{\epsilon/16}(n_{i+1}))}y_2,z_j)\\
				\overset{\eqref{s1}}\ge& 5\epsilon_0/4-(\epsilon/16)\times 2\ge9\epsilon_0/8.
			\end{split}
		\end{equation}
		Therefore, $E$ is an  $(\hat{n}_1,9\epsilon_0/8)$-separated set, and
		$\#E= \prod_{i=1}^{k}\#C_i(n_i,5\epsilon_0/4).$
		We note that
		\begin{align}
			&\prod_{i=1}^k\#C_i(n_i,5\epsilon_0/4)
			\overset{\eqref{5}}\ge\exp\left\{\sum_{i=1}^kn_i\left(\inf_{\xi\succ \mathcal{U}}h_{\nu_i}(f,\xi)-\gamma/2\right)\right\}\notag\\
			=&\exp\left\{\sum_{i=1}^k\frac{n_i}{\lambda_i}\left(\lambda_i\inf_{\xi\succ \mathcal{U}}h_{\nu_i}(f,\xi)-\frac{\lambda_i\gamma}{2}\right)\right\}
			\ge\exp\left\{\sum_{i=1}^kn\left(\lambda_i\inf_{\xi\succ \mathcal{U}}h_{\nu_i}(f,\xi)-{\lambda_i\gamma}\right)\right\},\label{q1}
		\end{align}provided by $n$ sufficiently large and notice that $\inf_{\xi\succ \mathcal{U}}h_{\nu_i}(f,\xi)<\infty$ for $i=1,...,k$ by \eqref{eq leq leq infty}. Furthermore, by \eqref{24} and $\delta\in(0,\gamma/2)$, we have
		\begin{equation*}
			\begin{split}
				\#E =\prod_{i=1}^k\#C_i(n_i,5\epsilon_0/4)
				\overset{\eqref{24},\eqref{q1}}\ge&\exp\left\{n\left(\inf_{\xi\succ \mathcal{U}}h_{\mu_1}(f,\xi)-\gamma-\delta\right)\right\}\\
				\ge&\exp\left\{n\left(\inf_{\xi\succ \mathcal{U}}h_{\mu_1}(f,\xi)-\frac{3\gamma}{2}\right)\right\}\\
				\overset{\eqref{pick n 2 Lepsilon}}\ge&\exp\left\{\hat n_1\left(\inf_{\xi\succ \mathcal{U}}h_{\mu_1}(f,\xi)-2\gamma\right)\right\}\\
				\ge& \exp\left\{\hat n_1\left(\inf_{|\xi|<5\epsilon_0}h_{\mu_1}(f,\xi)-2\gamma\right)\right\},
			\end{split}
		\end{equation*}where the last inequality is due to the fact that $\xi\succ \mathcal{U}$ implies $|\xi|<5\epsilon_0$.
		
		Now we only need to show that $E\subset P(\alpha_1,4\delta,\hat{n}_1)$. For any $y=y(x_1,...,x_k)\in E$, one has
		\begin{equation*}
			\begin{split}
				&\left|\sum_{j=0}^{\hat{n}_1-1}\varphi(f^j(y))-\hat{n}_1\alpha_1\right|\\
				\leq&\left| \sum_{p=0}^{n_1-1}\varphi(f^p(y))-\sum_{p=0}^{n_1-1}\varphi(f^p(x_1))+\sum_{p=0}^{n_1-1}\varphi(f^p(x_1))-n_1\int\varphi d\nu_1+n_1\int\varphi d\nu_1-n_1\alpha_1\right. \\
				&\quad+\cdots+\sum_{p=\sum_{i=1}^{k-1}(n_i+L_{\epsilon/16}(n_{i+1}))}^{\hat{n}_1-1}\varphi(f^p(y))-\sum_{p=0}^{n_k-1}\varphi(f^p(x_k))+\sum_{p=0}^{n_k-1}\varphi(f^p(x_k))-n_k\int\varphi d\nu_k\\
				&\quad\quad\left.+n_k\int\varphi d\nu_k-n_k\alpha_1\right|+2\sum_{i=1}^{k-1}L_{\epsilon}(n_{i+1})\|\varphi\|_{C^0}\\
				\leq &\left|\sum_{p=0}^{n_1-1}\varphi(f^p(y))-\sum_{p=0}^{n_1-1}\varphi(f^p(x_1))\right|+\left|\sum_{p=0}^{n_1-1}\varphi(f^p(x_1))-n_1\int\varphi d\nu_1\right|\\
				&\quad+\cdots +\left|\sum_{p=\sum_{i=1}^{k-1}(n_i+L_{\epsilon/16}(n_{i+1}))}^{\hat{n}_1-1}\varphi(f^p(y))-\sum_{p=0}^{n_k-1}\varphi(f^p(x_k))\right|+\left|\sum_{p=0}^{n_k-1}\varphi(f^p(x_k))-n_k\int\varphi d\nu_k\right|\\
				&\quad\quad+\left|n_1\int\varphi d\nu_1+\cdots +n_k\int\varphi d\nu_k-(n_1+\cdots +n_k)\alpha_1\right|+2\sum_{i=1}^{k-1}L_{\epsilon}(n_{i+1})\|\varphi\|_{C^0},
			\end{split}
		\end{equation*}where we employ the fact $|\varphi-\alpha_1|\leq 2\|\varphi\|_{C^0}$. By \eqref{s1}, for $ j=1,...,k, $ we have
		\begin{equation*}
			\left|\sum_{p=\sum_{i=1}^{j-1}(n_i+L_{\epsilon/16}(n_{i+1}))}^{\sum_{i=1}^{j-1}(n_i+L_{\epsilon/16}(n_{i+1}))+n_j-1}\varphi(f^p(y))-\sum_{p=0}^{n_j-1}\varphi(f^p(x_j))\right|\leq n_j\cdot var(\varphi,\frac{\epsilon}{16})  \overset{\eqref{pick epsilon}}\leq n_j\delta,
		\end{equation*}
		where $n_0=0$ and $var(\varphi,\frac{\epsilon}{16})=\sup\{|\varphi(w)-\varphi(z)|:\ d(w,z)<\frac{\epsilon}{16}\}$.
		Note that $x_j\in C_j(n_j,5\epsilon_0/4)\subset Y_j(N_0)$ and $n_j>N_0$ by \eqref{18}, then
		\begin{equation*}
			\left| \sum_{p=0}^{n_j-1}\varphi(f^p(x_j))-n_j\int\varphi d\nu_j\right|\leq n_j\delta\mbox{ for }j=1,...,k.
		\end{equation*}Moreover, note that $|\alpha_1|\le \|\varphi\|_{C^0}$, then
		\begin{align*}
			&|n_1\int\varphi d\nu_1+\cdots +n_k\int\varphi d\nu_k-(n_1+\cdots +n_k)\alpha_1|\\
			\leq & |n\int\varphi d\nu-n\alpha_1|+2(n-\sum_{i=1}^{k}n_i)\|\varphi\|_{C^0}\\
			\overset{\eqref{25}}\leq & n\delta+2k\|\varphi\|_{C^0}
			\overset{\eqref{pick n}}\leq \frac{3}{2}n\delta.
		\end{align*}Taking the second inequality in \eqref{pick n} into account, we arrive
		\begin{equation*}
			\frac{1}{\hat{n}_1} \left|\sum_{j=0}^{\hat{n}_1-1}\varphi(f^j(y))-\hat{n}_1\alpha_1\right|< 4\delta.
		\end{equation*}
		Thus, we have $E\subset P(\alpha_1,4\delta,\hat{n}_1)$. The proof of Lemma \ref{lem3} is completed.
	\end{proof}
	With the help of Lemma \ref{lem3}, we begin to construct the Moran-like fractal.
	Let $\rho:\mathbb{N}\to \{1,2\}$ be given by $\rho(k)=(k+1 \pmod2) +1$. Let $\{\delta_k\}_{k=1}^\infty$ be a strictly decreasing sequence such that $\delta_k\to 0$ as $k\to \infty$ and $\delta_1<\gamma$. Let $L_{\epsilon_0/2^{5+k}}:\mathbb{N}\to\mathbb{N}$ be  the tempered function in Definition \ref{def almostweakspecificationproperty} corresponding to $\epsilon_0/2^{5+k}$.  For each $k\in\mathbb{Z}^+$, there exists an constant $V_k\in \mathbb{N}$ such that for any $n\geq V_k$, one has
	\begin{equation}\label{Pick Vk1}
		\frac{\log(L_{\epsilon_0/2^{5+k}}(n))}{n}<\frac{\gamma}{2^{k+1}}.
	\end{equation}
	For each $k\in\mathbb{N}$, by Lemma \ref{lem3}, there exists $\hat{n}_k>V_k$ such that
	\begin{equation*}
		P(\alpha_{\rho(k)},4\delta_k,\hat{n}_k)=\left\{x\in X:\left|\frac{1}{\hat{n}_k}\sum_{j=0}^{\hat{n}_k-1}\varphi(f^j(x))-\alpha_{\rho(k)}\right|<4\delta_k\right\}\neq\emptyset
	\end{equation*}and
	\begin{equation}\label{9}
		\begin{split}
			M_k:=M(\alpha_{\rho(k)},4\delta_k,\hat{n}_k,9\epsilon_0/8)&\ge \exp\left\{\hat{n}_k\left(\inf _{|\xi|<5\epsilon_0} h_{\mu_{\rho(k)}}(f, \xi)-2\gamma\right)\right\}\\
			&\overset{\eqref{eq 3.14}}\ge\exp\left\{\hat{n}_k ((S-\gamma)|\log{5\epsilon_0}|-2\gamma)\right\} \\
			&\overset{\eqref{eq 3.15}}{\ge}\exp\left\{\hat{n}_k (S-3\gamma)|\log{5\epsilon_0}|\right\}.
		\end{split}
	\end{equation}
	Let $S_k$ be a maximal $(\hat{n}_k,9\epsilon_0/8)$-separated set in $P(\alpha_{\rho(k)},4\delta_k,\hat{n}_k)$ with $\#S_k=M_k$. For the notation simplification, for each $k\in\mathbb{N}$, we denote
	\begin{equation*}
		L_k:=L_{\epsilon_0/2^{5+k}}(\hat{n}_k).
	\end{equation*} Choose an integer sequence $\{N_k\}_{k=1}^\infty$ such that
	\begin{equation}\label{8}
		\lim_{k\to\infty}\frac{\hat{n}_{k+1}+L_{k+1}}{N_k}=0
	\end{equation}
	and
	\begin{equation}\label{pick Nk+1}
		\lim_{k\to\infty}\frac{\sum_{i=1}^{k}N_i(\hat n_i+L_i)}{N_{k+1}}=0.
	\end{equation}
	
	Now we construct the Moran-like fractal. The construction is a 3-step process. We describe the main difference between our construction and the construction in \cite{Thompson2010irregularspecification,TV,LF}, which is caused by the tempered gap function. Since the tempered gap function relies on the length of the next orbit segment, we cannot construct centers of $k$-th level set of the Moran-like fractal by gluing centers of $(k-1)$-th level set and $k$-th intermediate set as \cite{Thompson2010irregularspecification}, which makes the gap unexpected large. Therefore, we omit the construction of intermediate sets, and construct centers  of the $k$-th level set directly.
	
	\textbf{Step 1.}
	Construct $\mathcal{H}_1$, the center of the first level set of the Moran-like fractal.
	
	For each $N_1$-tuple  $(x_1^1,\ldots,x_{N_1}^1)\in{S}_1^{N_1}$, by using the almost weak specification property, there exists a point $y:=y(x_1^1,\ldots,x_{N_1}^1)$ such that
	\begin{equation}\label{shadowing y1}
		d_{\hat{n}_1}(f^{(p-1)(\hat{n}_1+L_1)}y, x_p^1)<\epsilon_0/2^{5+1}, \text{ for }p=1,2,\ldots,N_1.
	\end{equation}
	We collect all shadowing points into
	$$
	\mathcal{H}_1:=\left\{y(x_1^1,\ldots,x_{N_1}^1) \in X:(x_1^1,\ldots,x_{N_1}^1)\in {S}_1^{N_1}\right\} .
	$$
	Denote $t_1=N_1 \hat{n}_1+\left(N_1-1\right)L_1$. Then $t_1$ is the amount of time for which the orbit of points in $\mathcal{H}_1$ has been prescribed. The following lemma shows that distinct points $(x_1^1,\ldots,x_{N_1}^1)\in{S}_1^{N_1}$ give rise to distinct points in $\mathcal{H}_1$. Thus
	\begin{equation}\label{number H1}
		\#\mathcal{H}_1=\#{S}_1^{N_1}.
	\end{equation}
	\begin{lemma}\label{lem10}
		For any two different tuples $(x_1^1,\ldots,x_{N_1}^1),(z_1^1,\ldots,z_{N_1}^1)\in{S}_1^{N_1}$,  $y_1:=y(x_1^1,\ldots,x_{N_1}^1)$ and $y_2:=y(z_1^1,\ldots,z_{N_1}^1)$ are $\left(t_1, 17\epsilon_0/16\right)$-separated points, i.e. $d_{t_1}\left(y_1, y_2\right)>17\epsilon_0/16$.
	\end{lemma}
	\begin{proof}[Proof of Lemma \ref{lem10}]
		Denote $a_j=(j-1)(\hat{n}_1+L_{\epsilon_0/2^{5+1}}(\hat{n}_1))$ for each $j\in\{1,2,\ldots,N_1\}$.
		Since $(x_1^1,\ldots,x_{N_1}^1)\neq(z_1^1,\ldots,z_{N_1}^1)$, there exists $j\in\{1,2,\ldots,N_1\}$ such that $x_j^1\neq z_j^1$. Thus,
		\begin{equation*}
			\begin{split}
				d_{t_1}\left(y_1, y_2\right)&\ge d_{\hat{n}_1}(f^{a_j}y_1,f^{a_j}y_2)\\
				&\ge d_{\hat{n}_1}(x^1_j,z^1_j)-d_{\hat{n}_1}(f^{a_j}y_1,x^1_j)-d_{\hat{n}_1}(f^{a_j}y_2,z^1_j)\\
				&> 9\epsilon_0/8-\epsilon_0/2^{5+1}-\epsilon_0/2^{5+1}\\
				&\ge 17\epsilon_0/16,
			\end{split}
		\end{equation*}
		which finishes the proof of Lemma \ref{lem10}.
	\end{proof}

	\textbf{Step 2.}
	Construct $\mathcal{H}_k$, the center of the $k$-th level set of the Moran-like fractal.
	
	We define $\mathcal{H}_k$ inductively. Let $\mathcal{H}_1$ and $t_1$ be constructed as in Step 1. Suppose $\mathcal{H}_{k-1}$ and $t_{k-1}$  for $k\geq 2$ has been constructed, where $t_{k-1}$ can be viewed as the amount of time for which the orbit of points in $\mathcal{H}_{k-1}$ has been prescribed. In the following, we construct $\mathcal{H}_{k}$ and $t_{k}$.
	
	For any $(x,x_1^{k},\ldots,x_{N_{k}}^{k})\in\mathcal{H}_{k-1}\times {S}_k^{N_k}$, by the almost weak specification property, there exists a point $z:=z(x,x_1^{k},\ldots,x_{N_{k}}^{k})\in X$ such that
	\begin{equation}\label{s2}
		d_{t_{k-1}}(z,x)<\epsilon_0/2^{5+k} \text{ and }d_{\hat{n}_k}(f^{t_{k-1}+L_k+(i-1)(\hat{n}_k+L_k)}z,x_i^k)<\epsilon_0/2^{5+k}\mbox{ for }i=1,...,N_k.
	\end{equation}
	In this way, we call that $z=z(x,y)\in\mathcal{H}_{k}$ descends from
	$x\in\mathcal{H}_{k-1}$.
	Define  $$\mathcal{H}_{k}=\left\{z(x,x_1^{k},\ldots,x_{N_{k}}^{k})\in X: \ (x,x_1^{k},\ldots,x_{N_{k}}^{k})\in\mathcal{H}_{k-1}\times {S}_k^{N_k}\right\},$$ and
	\begin{equation}\label{number tk}
		t_{k}=t_{k-1}+N_{k}(\hat{n}_k+L_k).
	\end{equation}  By \eqref{s2}, the following lemma can be proved just like Lemma \ref{lem10}.
	\begin{lemma}\label{lem11}
		For every $x \in \mathcal{H}_{k-1}$ and distinct $N_k$-tuples $(x_1^{k},\ldots,x_{N_{k}}^{k}),(w_1^{k},\ldots,w_{N_{k}}^{k})\in {S}_k^{N_k}$, we denote $z_1=z(x,x_1^{k},\ldots,x_{N_{k}}^{k})$ and $z_2=z(x,w_1^{k},\ldots,w_{N_{k}}^{k})$. Then we have
		\begin{equation*}
			d_{t_{k-1}}\left(z_1, z_2\right)<\frac{\epsilon_0}{2^{5+k-1}} \text { and } d_{t_{k}}\left(z_1, z_2\right)>\frac{9\epsilon_0}{8}-2\times \frac{\epsilon_0}{2^{5+k}}>  \frac{17\epsilon_0}{16}.
		\end{equation*}
	\end{lemma}
	As a corollary of Lemma \ref{lem11}, \eqref{number H1} and the induction step, we have
	\begin{equation}\label{22}
		\# \mathcal{H}_k=\# S_1^{N_1} \ldots \# S_k^{N_k}=\prod_{i=1}^{k}M_i^{N_i},
	\end{equation}and any two points in $\mathcal{H}_k$ are $\left(t_k, \frac{17\epsilon_0}{16}\right)$-separated. In particular, if $z, z' \in \mathcal{H}_k$, then
	$$
	\overline{B}_{t_k}\left(z, \frac{\epsilon_0}{2^{5+k}}\right) \cap \overline{B}_{t_k}\left(z', \frac{\epsilon_0}{2^{5+k}}\right)=\emptyset,
	$$where $\overline{B}_{t_k}\left(z, \frac{\epsilon_0}{2^{5+k}}\right)=\{y\in X:\ d_{t_k}(z,y)\leq \frac{\epsilon_0}{2^{5+k}}\}$.
	\begin{lemma}\label{lem12}
		If $z\in \mathcal{H}_{k}$  descends from $x \in \mathcal{H}_{k-1}$, then
		\begin{equation}\label{sub}
			\overline{B}_{t_{k}}\left(z, \frac{\epsilon_0}{2^{5+k}}\right) \subset \overline{B}_{t_{k-1}}\left(x, \frac{\epsilon_0}{2^{5+k-1}}\right).
		\end{equation}
	\end{lemma}
	\begin{proof}[Proof of Lemma \ref{lem12}]
		By \eqref{s2}, $d_{t_{k-1}}(z,x)\leq \epsilon_0/2^{5+k}$. Thus,
		for any $w\in \overline{B}_{t_{k}}\left(z, \frac{\epsilon_0}{2^{5+k}}\right)$, one has
		\begin{equation*}
			\begin{split}
				d_{t_{k-1}}(w,x)\leq d_{t_{k-1}}(w,z)+d_{t_{k-1}}(z,x)\leq \frac{\epsilon_0}{2^{5+k}}+\frac{\epsilon_0}{2^{5+k}}= \frac{\epsilon_0}{2^{5+k-1}},
			\end{split}
		\end{equation*}
		which implies $w\in \overline{B}_{t_{k-1}}\left(x, \frac{\epsilon_0}{2^{5+k-1}}\right)$. Therefore, $
		\overline{B}_{t_{k}}\left(z, \frac{\epsilon_0}{2^{5+k}}\right) \subset \overline{B}_{t_{k-1}}\left(x, \frac{\epsilon_0}{2^{5+k-1}}\right)
		$.
	\end{proof}
	
	\textbf{Step 3.} Finish the construction of the Moran-like fractal and show that it is contained in $I_{\varphi}$.
	
	Let $\mathcal{X}_k:=\cup_{x\in \mathcal{H}_k}\overline{B}_{t_{k}}\left(x, \frac{\epsilon_0}{2^{5+k}}\right)$ for each $k\in\mathbb{N}$ and $\mathcal{X}:=\cap_{k\in\mathbb{N}}\mathcal{X}_k$. By Lemma \ref{lem12}, one has $\mathcal{X}_{k+1}\subset\mathcal{X}_k$ for each $k\in\mathbb{N}$. Then according to the compactness of $X$, $\mathcal{X}$ is a nonempty closed subset of $X$. The following lemma shows that the Moran-like fractal $\mathcal{X}$ is contained in $I_\varphi$.
	\begin{lemma}\label{lem13}
		For any $x\in\mathcal{X}$, the sequence $\{\frac{1}{t_k}\sum_{i=0}^{t_k-1}\phi(f^i(x))\}_{k\in\mathbb{N}}$ diverges. In particular, $\mathcal{X}\subset I_\varphi.$
	\end{lemma}
	\begin{proof}[Proof of Lemma \ref{lem13}]
		For any $x\in\mathcal{X}$, we prove Lemma \ref{lem13} by showing
		\begin{equation}\label{eq lemma 13}
			\lim_{k \rightarrow \infty}\left|\frac{1}{t_k} \sum_{j=0}^{t_k-1} \varphi(f^j(x))-\alpha_{\rho(k)}\right|=0,
		\end{equation} since  $\rho(k)$ is alternating between $1$ and $2$. Thus, we need  to estimate $\left|\sum_{j=0}^{t_k-1} \varphi(f^j(x))-t_k\alpha_{\rho(k)}\right|$, which is a 2-step process.
		
		\textbf{Step 1.}  Estimation on $\mathcal{H}_k$ for $k\geq 2$.
		
		For $k\geq 2$, let us estimate
		\begin{equation*}
			R_k=\max_{z\in \mathcal{H}_k}\left|\sum_{p=0}^{t_k-1}\varphi(f^p(z))-t_k\alpha_{\rho(k)}\right|.
		\end{equation*}For any $z\in \mathcal{H}_{k}$, there exists $(x,x_1^{k},\ldots,x_{N_{k}}^{k})\in\mathcal{H}_{k-1}\times {S}_k^{N_k}$ the satisfying shadowing property \eqref{s2}.
		On intervals of integers $[0,t_{k-1}+L_k-1]$ and $\cup_{i=1}^{N_k-1}[t_{k-1}+i(\hat{n}_k+L_k),t_{k-1}+i(\hat{n}_k+L_k)+L_k-1]$, we use $|\varphi-\alpha_{\rho(k)}|\leq 2\|\varphi\|_{C^0}$, while on intervals of integers $\cup_{i=1}^{N_k}[t_{k-1}+L_k+(i-1)(\hat{n}_k+L_k),t_{k-1}+i(\hat{n}_k+L_k)-1]$, we use the shadowing property in \eqref{s2} and the fact that $\{x_i^k\}_{i=1}^{N_k}\in P(\alpha_{\rho(k)},4\delta_k,\hat{n}_k)$ to obtain
		\begin{equation*}
			\begin{split}
				&\left|\sum_{p=0}^{t_k-1}\varphi(f^p(z))-t_k\alpha_{\rho(k)}\right|\\
				\leq & \sum_{i=1}^{N_k}\left(\left|\sum_{p=t_{k-1}+L_k+(i-1)(\hat{n}_k+L_k)}^{t_{k-1}+i(\hat{n}_k+L_k)-1}\varphi(f^p(z))-\sum_{p=0}^{\hat{n}_k-1}\varphi(f^p(x_i^k))\right|+\left|\sum_{p=0}^{\hat{n}_k-1}\varphi(f^p(x_i^k))-\hat{n}_k\alpha_{\rho(k)}\right|\right)\\
				&\quad+2(t_{k-1}+L_k)\|\varphi\|_{C^0}+2(N_k-1)L_k\|\varphi\|_{C^0}\\
				\leq & N_k\hat{n}_k (var(\varphi,\frac{\epsilon_0}{2^{5+k}})+4\delta_k)+2(t_{k-1}+N_k\cdot L_k)\|\varphi\|_{C^0}\\
				< &t_k (var(\varphi,\frac{\epsilon_0}{2^{5+k}})+4\delta_k)+2(t_{k-1}+N_k\cdot L_k)\|\varphi\|_{C^0}.
			\end{split}
		\end{equation*}We note that
		\begin{align*}
			\frac{t_{k-1}}{t_k}\overset{\eqref{number tk}}\leq \frac{\sum_{i=1}^{k-1}N_i(\hat{n}_i+L_i)}{N_k}\overset{\eqref{pick Nk+1}}\to 0 \text{ and }
			\frac{N_k\cdot L_k}{t_k}\overset{\eqref{number tk}}\leq  \frac{N_k\cdot L_k}{N_k\hat{n}_k}\overset{\eqref{Pick Vk1}}\to 0\mbox{ as }k\to \infty.
		\end{align*}Therefore, we have
		\begin{equation*}
			\frac{R_k}{t_k}\leq var(\varphi,\frac{\epsilon_0}{2^{5+k}})+4\delta_k+\frac{2(t_{k-1}+N_k\cdot L_k)\|\varphi\|_{C^0}}{t_k}\to 0, \text{ as }k\to \infty.
		\end{equation*}
		
		\textbf{Step 2.} Estimation on $\mathcal{X}$.
		
		For any $x\in\mathcal{X}$, $x\in \mathcal{X}_k=\cup_{z\in \mathcal{H}_k}\overline{B}_{t_k}(z,\frac{\epsilon_0}{2^{5+k}})$ for any $k\in\mathbb{N}$. Therefore, for any $k\in\mathbb{N}$ there exists  $z_k(x)\in \mathcal{H}_k$ such that $x\in \overline{B}_{t_k}(z_k(x),\frac{\epsilon_0}{2^{5+k}}) $. Now
		\begin{align*}
			& \lim_{k \rightarrow \infty}\left|\frac{1}{t_k} \sum_{j=0}^{t_k-1} \varphi(f^j(x))-\alpha_{\rho(k)}\right|\\
			\leq  &\lim_{k \rightarrow \infty}\left(\left|\frac{1}{t_k} \sum_{j=0}^{t_k-1} \varphi(f^j(x))-\frac{1}{t_k} \sum_{j=0}^{t_k-1} \varphi(f^j(z_k(x)))\right|+\left|\frac{1}{t_k} \sum_{j=0}^{t_k-1} \varphi(f^j(z_k(x)))-\alpha_{\rho(k)}\right|\right)\\
			\leq &\lim_{k \rightarrow \infty}\left(var\left(\varphi,\frac{\epsilon_0}{2^{5+k}}\right)+\frac{R_k}{t_k}\right)\\
			=&0.
		\end{align*}
		The proof of Lemma \ref{lem13} is completed.
	\end{proof}
	
	\subsection{Construct a suitable measure on this fractal}\label{subsection 3.2}
	In this subsection, we construct a suitable measure on the fractal $\mathcal{X}$, and then apply entropy distribution principle type argument.

	For each $k\in\mathbb{N}$, define the probability measure
	\[\nu_k:=\frac{1}{\#\mathcal{H}_k}\sum_{x\in\mathcal{H}_k}\delta_x.\]
	By the compactness of $M(X)$, there exists a subsequence of $\{\nu_k\}_{k=1}^\infty$ converges to a measure in $M(X)$ with respect to the weak*-topology, named $\nu_{k_l}\to\nu$ as $l\to \infty$.

	By the construction of $\nu_k$, we have $\nu_k(\mathcal{X}_k)=1$ for each $k\in\mathbb{N}$. Since $\mathcal{X}_{k_l}\subset\mathcal{X}_{k}$ if $k_l\geq k$, it follows that
	\[\nu_{k_l}(\mathcal{X}_k)=1,\text{ for any } k_l\geq k.\]
	As $\mathcal{X}_k$ is closed, one has
	$\nu(\mathcal{X}_k)\ge\limsup_{l\to\infty}\nu_{k_l}(\mathcal{X}_k)=1,$
	which implies that \begin{equation}\label{meausre}
		\nu(\mathcal{X})=\nu(\cap_{k=1}^\infty\mathcal{X}_k)=1.
	\end{equation} We need the following lemma to apply the entropy distribution principle type argument.
	\begin{lemma}\label{lem17}
		There exists $N\in\mathbb{N}$ such that for any $n\ge N$, if $B_n(x,\epsilon_0/4)\cap \mathcal{X}\neq\emptyset$, then
		\[\nu(B_n(x,\epsilon_0/4))\le \exp(-n(S-4\gamma)|\log5\epsilon_0|).\]
	\end{lemma}
	\begin{proof}[Proof of Lemma \ref{lem17}]
		By the choice of $N_k$ satisfying \eqref{8} and \eqref{pick Nk+1} and choice of  $\hat{n}_k\ge V_k$ satisfying \eqref{Pick Vk1}, there exists $K>0$ such that for any $k>K$ and $j\in\{0,1,...,N_{k+1}-1\}$,
		\begin{equation}\label{pick k1}
			\frac{N_1\hat{n}_1+\ldots+N_k\hat{n}_k+j\hat{n}_{k+1}}{N_1(\hat{n}_1+L_1)+\ldots+N_k(\hat{n}_k+L_k)+j(\hat{n}_{k+1}+L_{k+1})}\geq \frac{S-\frac{7\gamma}{2}}{S-3\gamma},
		\end{equation}and
		\begin{equation}\label{pick K}
			\frac{\hat{n}_{k+1}+L_{k+1}}{N_k}<\frac{\frac{\gamma}{2}}{S-\frac{7\gamma}{2}}.
		\end{equation}
		For any integer $n\geq N:=t_K$, there exist a unique $k^*\geq K$ and a unique	$j\in\{0,1,\ldots,N_{k^*+1}-1\}$ such that \begin{equation}\label{11}
			t_{k^*}+j\cdot(\hat{n}_{k^*+1}+L_{k^*+1})\le n<t_{k^*}+(j+1)\cdot(\hat{n}_{k^*+1}+L_{k^*+1}).
		\end{equation}
		Note that $B_n(x,\epsilon_0/4)$ is an open set, then
		\begin{equation}\label{estimate measure}
			\begin{split}
				\nu(B_n(x,\epsilon_0/4))&\le \liminf_{l\to\infty}\nu_{k_l}(B_n(x,\epsilon_0/4))\leq \limsup_{q\to\infty}\nu_{k^*+q}(B_n(x,\epsilon_0/4))\\ &=\limsup_{q\to\infty}\frac{1}{\#\mathcal{H}_{k^*+q}}\#(\mathcal{H}_{k^*+q}\cap B_n(x,\epsilon_0/4)).
			\end{split}
		\end{equation}
		Therefore, we wish to estimate $\frac{1}{\#\mathcal{H}_{k^*+q}}\#(\mathcal{H}_{k^*+q}\cap B_n(x,\epsilon_0/4))$. The estimation is a 3-step process.
		
		\textbf{Step 1.} We show that $\#(\mathcal{H}_{k^*}\cap B_n(x,\epsilon_0/2))\le1$.
		
		Since any two distinct points in $\mathcal{H}_{k^*}$ are $(t_{k^*},\frac{17\epsilon_0}{16})$-separated, which is  proved in Lemma \ref{lem11}, it follows that $\#(\mathcal{H}_{k^*}\cap B_n(x,\epsilon_0/2))\le\#(\mathcal{H}_{k^*}\cap B_{t_{k^*}}(x,\epsilon_0/2))\le1$.
		
		\textbf{Step 2.} We show that $\#(\mathcal{H}_{k^*+1}\cap B_n(x,\epsilon_0/2))\le (M_{k^*+1})^{N_{k^*+1}-j}$, where $j$ is given in \eqref{11}.
		
		The case for $j=0$ is obtained immediately by Step 1. Now we consider the case for $j\ge 1$. Let $z_1=z(y,y_1^{k^*+1},\ldots,y_{N_{k^*+1}}^{k^*+1}),z_2=z(w,w_1^{k^*+1},\ldots,w_{N_{k^*+1}}^{k^*+1})\in \mathcal{H}_{k^*+1}$ with  $(y,y_1^{k^*+1},\ldots,y_{N_{k^*+1}}^{k^*+1})$, $(w,w_1^{k^*+1},\ldots,w_{N_{k^*+1}}^{k^*+1})\in\mathcal{H}_{k^*}\times ({S}_{k^*+1})^{N_{k^*+1}}$.
		We claim that $y=w$ and $y^{k^*+1}_1=w^{k^*+1}_1, \ldots, y^{k^*+1}_j=w^{k^*+1}_j$. In fact, if $y\not=w$, then they are $(t_{k^*},17\epsilon_0/16)$-separated by Lemma \ref{lem11}, but we also have
		\begin{align*}
			d_{t_{k^*}}(y,w)\leq& d_{t_{k^*}}(y,z_1)+  d_{t_{k^*}}(z_1,x)+ d_{t_{k^*}}(x,z_2)+d_{t_{k^*}}(z_2,w)\\
			\leq  &d_{t_{k^*}}(y,z_1)+  d_{n}(z_1,x)+ d_{n}(x,z_2)+d_{t_{k^*}}(z_2,w)\\
			\overset{\eqref{s2}}  \leq &\frac{\epsilon_0}{2^{5+k^*+1}}+\frac{\epsilon_0}{2}+\frac{\epsilon_0}{2}+\frac{\epsilon_0}{2^{5+k^*+1}}\\
			\leq &33\epsilon_0/32,
		\end{align*}which leads to a contradiction. If there exists $p\in\{1,\ldots,j\}$ such that $y^{k^*+1}_{{p}}\neq w^{k^*+1}_{{p}}$, then $y^{k^*+1}_{p},w^{k^*+1}_{{p}}$ are $(\hat{n}_{k^*+1},{9\epsilon_0}/{8})$-separated, as they belong to $S_{k^*+1}$. However, by \eqref{s2}, we have
		\begin{align*}
			& d_{\hat{n}_{k^*+1}}(y^{k^*+1}_{{p}},w^{k^*+1}_{{p}})\\
			\leq  & d_{\hat{n}_{k^*+1}}(y^{k^*+1}_{{p}},f^{t_{k^*}+L_{k^*+1}+(p-1)(\hat{n}_{k^*+1}+L_{k^*+1})}z_1)+d_{n}(z_1,x)\\&+d_{n}(x,z_2)
			+d_{\hat{n}_{k^*+1}}(f^{t_{k^*}+L_{k^*+1}+(p-1)(\hat{n}_{k^*+1}+L_{k^*+1})}z_2,w^{k^*+1}_{{p}})\\
			\leq & \frac{\epsilon_0}{2^{5+k^*+1}}+\frac{\epsilon_0}{2}+\frac{\epsilon_0}{2}+\frac{\epsilon_0}{2^{5+k^*+1}}\\
			\leq &33\epsilon_0/32
		\end{align*}
		which leads to
		a contradiction. The statement of Step 2 is a direct corollary of the above claim.
		
		\textbf{Step 3.} We show  that \begin{equation}\label{19}
			\#(\mathcal{H}_{k^*+q}\cap B_n(x,\epsilon_0/4))\le M_{k^*+1}^{N_{k^*+1}-j}\prod_{i=2}^qM_{k^*+i}^{N_{k^*+i}} \text{ for any } q\ge 2.
		\end{equation}
		We claim that any $z_q\in\mathcal{H}_{k^*+q}\cap B_n(x,\epsilon_0/4)$ must descend from some point in $\mathcal{H}_{k^*+1}\cap B_n(x,\epsilon_0/2)$. In fact, as $z_q$ descends from  some $z\in\mathcal{H}_{k^*+1}$, then we can apply Lemma \ref{lem12} inductively to show that
		\begin{equation*}
			\bar{B}_{t_{k^*+q}}(z_q,\frac{\epsilon_0}{2^{5+k^*+q}})\subset \bar{B}_{t_{k^*+1}}(z,\frac{\epsilon_0}{2^{5+k^*+1}}).
		\end{equation*}
		It follows that
		\begin{equation*}
			\begin{split}
				d_n(z,x)\le& d_n(z,z_q)+d_n(z_q,x)\le d_{t_{k^*+1}}(z,z_q)+\epsilon_0/4\overset{\eqref{sub}}\le \frac{\epsilon_0}{2^{5+k^*+1}}+\epsilon_0/4<\epsilon_0/2.
			\end{split}
		\end{equation*}
		Thus, by Step 2 and the above claim,  \eqref{19} holds.
		
		Note that
		\begin{align*}
			& M_1^{N_1}\ldots M_{k^*}^{N_{k^*}}\cdot M_{k^*+1}^j \\
			\overset{\eqref{9}}\geq & \exp\left((N_1\hat{n}_1+\ldots+N_{k^*}\hat{n}_{k^*}+j\hat{n}_{k^*+1})(S-3\gamma)|\log5\epsilon_0|\right)\\
			\overset{\eqref{pick k1}}  \geq & \exp\left((N_1(\hat{n}_1+L_1)+\ldots+N_{k^*}(\hat{n}_{k^*}+L_{k^*})+j(\hat{n}_{k^*+1}+L_{k^*+1}))(S-\frac{7\gamma}{2})|\log5\epsilon_0|\right)\\
			\overset{\eqref{11}}\geq  &\exp\left((n-(\hat{n}_{k^*+1}+L_{k^*+1}))(S-\frac{7\gamma}{2})|\log5\epsilon_0|\right)\\
			=&\exp\left(n(1-\frac{\hat{n}_{k^*+1}+L_{k^*+1}}{n})(S-\frac{7\gamma}{2})|\log5\epsilon_0|\right)\\
			\overset{\eqref{pick K}}\geq &\exp\left(n(S-4\gamma)|\log5\epsilon_0|\right).
		\end{align*}
		Now, we have, for each $q\ge 2$,
		\begin{equation*}
			\begin{split}
				\nu_{k^*+q}(B_n(x,\epsilon_0/4))=&\frac{1}{\#\mathcal{H}_{k^*+q}}\#(\mathcal{H}_{k^*+q}\cap B_n(x,\epsilon_0/4))\overset{\eqref{19},\eqref{22}}\le \frac{1}{\#\mathcal{H}_{k^*}\cdot M_{k^*+1}^j}\\
				=&\frac{1}{M_1^{N_1}\ldots M_{k^*}^{N_{k^*}}\cdot M_{k^*+1}^j}\leq \exp(-n(S-4\gamma)|\log5\epsilon_0|).
			\end{split}
		\end{equation*}
		Therefore,
		\[\nu(B_n(x,\epsilon_0/4))\overset{\eqref{estimate measure}}\le \limsup_{q\to\infty}\nu_{k^*+q}(B_n(x,\epsilon_0/4))\le\exp\left(-n\left(S-4\gamma\right)|\log5\epsilon_0|\right).\]The proof of Lemma \ref{lem17} is completed.
	\end{proof}

	\subsection{Apply entropy distribution principle type argument}\label{subsection 3.3}
	Now we are able to finish the proof of Theorem \ref{main} by using the entropy distribution principle type argument.
	
	Let $N$ be the number defined in Lemma \ref{lem17}.
	Let $\Gamma=\{B_{n_i}(x_i,{\epsilon_0}/{4})\}_{i\in I}$ be any finite cover of $\mathcal{X}$ with $n_i\ge N$ for all $i\in I$. Here we only need to consider finite cover since $\mathcal{X}$ is compact. Without loss of generality, we may assume that $B_{n_i}(x_i,{\epsilon_0}/{4})\cap \mathcal{X}\neq \emptyset$ for every $i\in I$. Applying Lemma \ref{lem17} on each $B_{n_i}(x_i,{\epsilon_0}/{4})$, one has
	\[\sum_{i \in I}\exp(-n_i(S-4\gamma)|\log5\epsilon_0|)\ge\sum_{i \in I}\nu(B_{n_i}(x_i,{\epsilon_0}/{4}))\ge\nu(\mathcal{X})\overset{\eqref{meausre}}=1.\]
	As $\Gamma$ is arbitrary, one has
	$$m(\mathcal{X},( S-4\gamma)|\log5\epsilon_0|,N, \epsilon_0/4)\ge 1>0.$$
	Therefore, by the fact that $m(\mathcal{X},( S-4\gamma)|\log5\epsilon_0|,N, \epsilon_0/4)$ does not decrease as $N$ increases, $$m(\mathcal{X},( S-4\gamma)|\log5\epsilon_0|, \epsilon_0/4)\geq 1>0,$$ which implies that
	$$h_{\operatorname{top}}^B(\mathcal{X},f,\epsilon_0/4)\ge (S-4\gamma)|\log5\epsilon_0|.$$
	By Lemma \ref{lem13}, one has
	\begin{equation*}
		\begin{split}
			S-4\gamma\le& \frac{h_{\operatorname{top}}^B(\mathcal{X},f,\epsilon_0/4)}{|\log5\epsilon_0|}\le \frac{h_{\operatorname{top}}^B(I_{\varphi},f,\epsilon_0/4)}{|\log5\epsilon_0|}
			=\frac{h_{\operatorname{top}}^B(I_{\varphi},f,\epsilon_0/4)}{|\log\epsilon_0/4|}\cdot\frac{|\log\epsilon_0/4|}{|\log5\epsilon_0|}\\
			&\overset{\eqref{eq 3.12}}{\le}(\overline{\operatorname{mdim}}_\mathrm{M}^B(I_{\varphi}, f, d)+\gamma )\cdot\frac{|\log\epsilon_0/4|}{|\log5\epsilon_0|}\\
			&\overset{\eqref{eq 3.13}}\le\overline{\operatorname{mdim}}_\mathrm{M}^B(I_{\varphi}, f, d)+2\gamma.
		\end{split}
	\end{equation*}
	Thus, $\overline{\operatorname{mdim}}_\mathrm{M}^B(I_{\varphi}, f, d)\ge S-6\gamma$. As $\gamma>0$ is arbitrary, we obtain $S\le \overline{\operatorname{mdim}}_\mathrm{M}^B(I_{\varphi}, f, d).$ Now we have finished the the proof of
	$\overline{\operatorname{mdim}}_\mathrm{M}(X, f, d)= \overline{\operatorname{mdim}}_\mathrm{M}^B(I_{\varphi}, f, d).$

	\subsection{The case of the lower metric mean dimension}
	In this subsection, we briefly prove the following equation
	\begin{equation*}
		\underline{\operatorname{mdim}}_{\mathrm{M}}^B(I_\varphi, f, d)=\underline{\operatorname{mdim}}_{\mathrm{M}}(X, f, d)
	\end{equation*}under the assumptions $I_\varphi\not=\emptyset$ and $\underline{\operatorname{mdim}}_{\mathrm{M}}(X, f, d)<\infty$.
	
	Proposition \ref{propositiob coincides} and \eqref{eq subset leq} imply $ \underline{\operatorname{mdim}}_{\mathrm{M}}^B(I_\varphi, f, d)\leq \underline{\operatorname{mdim}}_{\mathrm{M}}(X, f, d)$. In the following, we prove $$ \underline{\operatorname{mdim}}_{\mathrm{M}}^B(I_\varphi, f, d)\geq \underline{\operatorname{mdim}}_{\mathrm{M}}(X, f, d).$$ Denote $S^\prime=\underline{\operatorname{mdim}}_{\mathrm{M}}(X, f, d)<\infty$, and we only need to consider the case that $S^\prime>0$. Fixing any sufficiently small $\gamma\in \left(0,\min\left\{{S'}/{7},1\right\}\right)$, we are going to show
	\begin{equation*}
		\underline{\operatorname{mdim}}_{\mathrm{M}}^B(I_\varphi, f, d)\geq S^\prime-6\gamma.
	\end{equation*}We replace Lemma \ref{lem2} by the following lemma.
	\begin{lemma}\label{lemma replace}
		There exists $\epsilon_1=\epsilon_1(\gamma)$ such that
		\begin{align}
			&|\log 5\epsilon_1|>1;\label{r11}\\
			& S^\prime-\gamma/2\le\frac{1}{|\log 5\epsilon_1|} \sup _{\mu \in M(X,f)} \inf _{|\xi|<5\epsilon_1} h_\mu(f, \xi);\label{eq r 1}\\
			&\frac{h_{\operatorname{top}}^B(I_{\varphi},f,\epsilon_1/4)}{|\log{\epsilon_1/4}|}\le\underline{\operatorname{mdim}}_\mathrm{M}^B(I_{\varphi}, f, d)+\gamma;\label{eq r2}\\
			&(\underline{\operatorname{mdim}}_\mathrm{M}^B(I_{\varphi}, f, d)+\gamma )\cdot\frac{|\log\epsilon_1/4|}{|\log5\epsilon_1|} \le\underline{\operatorname{mdim}}_\mathrm{M}^B(I_{\varphi}, f, d)+2\gamma.\label{eq r4}
		\end{align}
		Moreover, there exist $\mu_1, \mu_2 \in M\left(X,f\right)$ such that
		\begin{equation}\label{eq r5}
			\int \varphi d \mu_1 \neq \int \varphi d \mu_2\mbox{ and }\frac{1}{|\log 5\epsilon_1|} \inf _{|\xi|<5\epsilon_1} h_{\mu_i}(f, \xi)>S'-\gamma\mbox{ for $i=1,2$.}
		\end{equation}
	\end{lemma}
	\begin{proof}
		We first pick $\delta_1>0$ such that for any $\epsilon\in(0,\delta_1)$, one has
		\begin{equation}\label{eq r4 s4}
			(\underline{\operatorname{mdim}}_\mathrm{M}^B(I_{\varphi}, f, d)+\gamma )\cdot\frac{|\log\epsilon/4|}{|\log5\epsilon|}
			\le\underline{\operatorname{mdim}}_\mathrm{M}^B(I_{\varphi}, f, d)+2\gamma,\mbox{ and }|\log5\epsilon|>1.
		\end{equation}By Lemma \ref{lemma 2.7}, we pick $\delta_2\in (0,\delta_1)$ such that
		\begin{equation}\label{eq r1s1}
			S^\prime-\frac{\gamma}{2}\leq    \inf_{5\epsilon\in(0,\delta_2)}\frac{1}{|\log 5\epsilon|} \sup _{\mu \in M(X,f)} \inf _{|\xi|<5\epsilon} h_\mu(f, \xi).
		\end{equation}
		Finally, by \eqref{2}, we pick $\epsilon_1\in (0,\frac{\delta_2}{5})$ such that \eqref{eq r2} holds. Note that $5\epsilon_1\in (0,\delta_2)$, then \eqref{eq r 1} is a direct corollary of \eqref{eq r1s1}. Moreover, \eqref{r11} and \eqref{eq r4} are guaranteed by \eqref{eq r4 s4}. The proof of \eqref{eq r5} is similar as the proof of \eqref{eq 3.14}.
	\end{proof}
	We can use the parallel proof in the Subsection \ref{subsection 3.1} and Subsection \ref{subsection 3.2} to show that there exist a Moran-like fractal $\mathcal{X}^\prime\subset I_\varphi$ and a measure $\nu^\prime $ concentrated on $\mathcal{X}^\prime$ satisfying the following property.
	\begin{lemma}\label{lemma r3.8}
		There exists $N^\prime\in\mathbb{N}$ such that for any $n\geq N'$, if $B_n(x,\epsilon_1/4)\cap \mathcal{X}^\prime\not=\emptyset$, then
		\begin{equation*}
			\nu^\prime(B_n(x,\epsilon_1/4))\leq \exp(-n(S^\prime-4\gamma)|\log 5\epsilon_1|).
		\end{equation*}
	\end{lemma}
	Let $N^\prime$ be the number defined in Lemma \ref{lemma r3.8}.
	Let $\Gamma=\{B_{n_i}(x_i,{\epsilon_1}/{4})\}_{i\in I}$ be any finite cover of $\mathcal{X}^\prime$ with $n_i\ge N^\prime$ for all $i\in I$. Without loss of generality, we may assume that $B_{n_i}(x_i,{\epsilon_1}/{4})\cap \mathcal{X}^\prime\neq \emptyset$ for every $i\in I$. Applying Lemma \ref{lemma r3.8} on each $B_{n_i}(x_i,{\epsilon_1}/{4})$, one has
	\[\sum_{i \in I}\exp(-n_i(S^\prime-4\gamma)|\log5\epsilon_1|)\ge\sum_{i \in I}\nu^\prime(B_{n_i}(x_i,{\epsilon_1}/{4}))\ge\nu^\prime(\mathcal{X}^\prime)=1.\]
	As $\Gamma$ is arbitrary, one has
	$$m(\mathcal{X}^\prime,(S'-4\gamma)|\log5\epsilon_1|,N^\prime, \epsilon_1/4)\ge 1>0.$$
	Therefore,  by the fact that $m(\mathcal{X},( S-4\gamma)|\log5\epsilon_0|,N, \epsilon_0/4)$ does not decrease as $N$ increases, $$m(\mathcal{X}^\prime,(S'-4\gamma)|\log5\epsilon_1|,\epsilon_1/4)>0,$$ which implies that $$h_{\operatorname{top}}^B(\mathcal{X}^\prime,f,\epsilon_1/4)\ge (S^\prime-4\gamma)|\log5\epsilon_1|.$$
	By Lemma \ref{lem13}, one has
	\begin{equation*}
		\begin{split}
			S^\prime-4\gamma\le& \frac{h_{\operatorname{top}}^B(\mathcal{X}^\prime,f,\epsilon_1/4)}{|\log5\epsilon_1|}\le \frac{h_{\operatorname{top}}^B(I_{\varphi},f,\epsilon_1/4)}{|\log5\epsilon_1|}
			=\frac{h_{\operatorname{top}}^B(I_{\varphi},f,\epsilon_1/4)}{|\log\epsilon_1/4|}\cdot\frac{|\log\epsilon_1/4|}{|\log5\epsilon_1|}\\
			&\overset{\eqref{eq r2}}{\le}(\underline{\operatorname{mdim}}_\mathrm{M}^B(I_{\varphi}, f, d)+\gamma )\cdot\frac{|\log\epsilon_1/4|}{|\log5\epsilon_1|}\\
			&\overset{\eqref{eq r4}}\le\underline{\operatorname{mdim}}_\mathrm{M}^B(I_{\varphi}, f, d)+2\gamma.
		\end{split}
	\end{equation*}
	Thus, $\underline{\operatorname{mdim}}_\mathrm{M}^B(I_{\varphi}, f, d)\ge S^\prime-6\gamma$. As $\gamma>0$ is arbitrary, we obtain $S^\prime\le \underline{\operatorname{mdim}}_\mathrm{M}^B(I_{\varphi}, f, d).$ So we finish the proof of
	$\underline{\operatorname{mdim}}_\mathrm{M}(X, f, d)= \underline{\operatorname{mdim}}_\mathrm{M}^B(I_{\varphi}, f, d).$
	
	The proof of Theorem \ref{main} is completed.

	%%%%%%%%%%%%%%%%%%%%%%%%%%%%%%%%%%%%%%%%%%%%%%%%%%%%%%%%%%%%%%%%%%%%%%%%%%%%%%%%%%%%%%
	\section{Application and Examples}\label{section 4}
	In this section, we discuss some examples and applications, and we only consider the case for the upper metric dimension, while the case for the lower metric dimension is similar.
	
	The following example is the main example that Theorem \ref{main} can be applied.
	\begin{example}
		Let $(K,d)$ be any compact metric space with the upper box-counting dimension $\overline{dim}_B(K)<\infty$. Then $\overline{mdim}_M(K^\mathbb{Z},\sigma ,d^\prime)=\overline{dim}_B(K,d)<\infty$ (see e.g. \cite[Corollary 16]{S}). We consider left-shift map $\sigma$ on the compact metric space $K^\mathbb{Z}$ equipped with the metric
		\begin{equation*}
			d^\prime((x_n)_{n\in\mathbb{Z}},(y_n)_{n\in\mathbb{Z}}):=\sum_{n\in\mathbb{Z}}\frac{d(x_n,y_n)}{2^{|n|}}
		\end{equation*}for any $(x_n)_{n\in\mathbb{Z}},(y_n)_{n\in\mathbb{Z}}\in K^\mathbb{Z}$. By \cite[Proposition 21.2]{DGS}, system $(K^\mathbb{Z},\sigma)$ has the specification property and therefore  the almost weak specification property.
	\end{example}

	Next, we discuss an application of Theorem \ref{main}.
	Let $(X,d,f)$ and $(Y,\rho,g)$ be any two TDSs. On the product space $X\times Y$, we consider the metric
	\[(d\times \rho)((x,y),(x',y'))=d(x,x')+\rho(y,y')\text{ for any }x,x'\in X\text{ and }y,y'\in Y.\]  Similar to the box dimension, Acevedo \cite[Theorem 3.7]{Ac} showed that
	\[\overline{\operatorname{mdim}}_\mathrm{M}(X\times Y, f\times g, d\times \rho)\le \overline{\operatorname{mdim}}_\mathrm{M}(X, f, d)+\overline{\operatorname{mdim}}_\mathrm{M}(Y, g, \rho).\]
	Note that if $f$ and $g$ have the almost weak specification property, $f\times g$ also has the almost weak specification property (see \cite[Lemma 2.2]{Sun2017}). As an application of Theorem \ref{main}, we immediately obtain the following corollary.
	\begin{corollary}
		Let $(X,f,d)$ and $(Y,g,\rho)$ be two TDSs with the almost weak specification property. Assume that $\overline{\operatorname{mdim}}_\mathrm{M}(X, f, d)<\infty$ and $\overline{\operatorname{mdim}}_\mathrm{M}(Y, g, \rho)<\infty$. Then for any two continuous observables $\varphi\in C(X,\mathbb{R})$ and $\psi\in C(Y,\mathbb{R})$, one has
		\[\overline{\operatorname{mdim}}_\mathrm{M}^B(I_{\varphi\times\psi}, f\times g, d\times \rho)\le \overline{\operatorname{mdim}}_\mathrm{M}^B(I_\varphi, f, d)+\overline{\operatorname{mdim}}_\mathrm{M}^B(I_\psi, g, \rho).\]
	\end{corollary}

	The following example shows that the above inequality can be strict.
	\begin{example}
		Authors in \cite[Theorem 3]{WWW} showed that for any  $\beta, \gamma, \lambda>0$ with  $\lambda \leq \gamma+\beta$, there are two compact metric spaces $(E,d_E)$ and $(F,d_F)$ satisfying
		$$
		\overline{\operatorname{dim}}_B (F,d_F)=\beta, \text{ }
		\overline{\operatorname{dim}}_B (E,d_E)=\gamma, \text{ and } \overline{\operatorname{dim}}_B(E \times F,d_E\times d_F)=\lambda.
		$$
		In particular, there are two compact metric spaces $(X,d)$ and  $(Y,\rho)$ such that
		\begin{equation}\label{m3}
			\overline{\operatorname{dim}}_B(X\times Y,d\times \rho)<\overline{\operatorname{dim}}_B(X,d)+\overline{\operatorname{dim}}_B(Y,\rho)<\infty.
		\end{equation}Fix any two compact metric spaces $(X,d)$ and $(Y,\rho)$ satisfying \eqref{m3}.
		Define The metric $\tilde{d} \times \tilde{\rho}$ on $X^\mathbb{Z} \times Y^\mathbb{Z}$ by
		$$
		(\widetilde{d} \times \widetilde{\rho})((\bar{x}, \bar{y}),(\bar{z}, \bar{w}))=\sum_{i \in\mathbb{ Z}} \frac{1}{2^{|i|}} d\left(x_i, z_i\right)+\sum_{i \in \mathbb{Z}} \frac{1}{2^{|i|}} \rho\left(y_i, w_i\right),
		$$
		for $\bar{x}=\left(x_i\right)_{i \in \mathbb{Z}}, \bar{z}=\left(z_i\right)_{i \in \mathbb{Z}} \in X^\mathbb{Z}$ and  $\bar{y}=\left(y_i\right)_{i \in \mathbb{Z}}, \bar{w}=\left(w_i\right)_{i \in \mathbb{Z}} \in Y^\mathbb{Z}$. Let $\sigma_1$ and $\sigma_2$ be the left-shift maps on $X^\mathbb{Z}$ and $Y^\mathbb{Z}$, respectively. Then both of them have the specification property. Note that
		\begin{equation}\label{m1}
			\overline{\operatorname{mdim}}_{\mathrm{M}}\left(X^\mathbb{Z} ,\widetilde d, \sigma_1 \right)=\overline{\operatorname {dim}}_B(X,d),\text{ }\overline{\operatorname{mdim}}_{\mathrm{M}}\left(Y^\mathbb{Z} ,\widetilde \rho, \sigma_2 \right)=\overline{\operatorname {dim}}_B(Y,\rho),
		\end{equation}
		and
		\begin{equation}\label{m2}
			\overline{\operatorname{mdim}}_{\mathrm{M}}\left(X^\mathbb{Z} \times Y^\mathbb{Z}, \widetilde{d} \times \widetilde{\rho}, \sigma_1 \times \sigma_2\right)=\overline{\operatorname {dim}}_B(X\times Y,d\times\rho).
		\end{equation}
		Thus, by \eqref{m3}, \eqref{m1} and \eqref{m2}, one has
		\[\overline{\operatorname{mdim}}_{\mathrm{M}}\left(X^\mathbb{Z} \times Y^\mathbb{Z}, \widetilde{d} \times \widetilde{\rho}, \sigma_1 \times \sigma_2\right) <\overline{\operatorname{mdim}}_{\mathrm{M}}\left(X^{\mathbb{Z}}, \widetilde{d}, \sigma_1\right)+\overline{\operatorname{mdim}}_{\mathrm{M}}\left(Y^\mathbb{Z}, \widetilde{\rho}, \sigma_2\right).\]
		Taking Theorem \ref{main} and Proposition \ref{propositiob coincides} into account, we conclude that for any two continuous observables $\varphi: X\to \mathbb{R}$ and $\psi:Y\to\mathbb{R}$, one has \[\overline{\operatorname{mdim}}_\mathrm{M}^B(I_{\varphi\times\psi}, \sigma_1\times \sigma_2, \widetilde d\times \widetilde\rho)< \overline{\operatorname{mdim}}_\mathrm{M}^B(I_\varphi, \sigma_1,\widetilde d)+\overline{\operatorname{mdim}}_\mathrm{M}^B(I_\psi, \sigma_2, \widetilde\rho).\]
	\end{example}
	
	%\section*{Acknowledgement}
	%C. Liu was partially supported by NNSF of China (12090012).
	
	%%%%%%%%%%%%%%%%%%%%%%%%%%%%%%%%%%%%%%%%%%%%%%%%%%%%%%%%%%%%%%%%%%%%%%
	\appendix
	\section{Some proof.}\label{section appendix}\label{appendix}
	\begin{proof}[Proof of Proposition \ref{propositiob coincides}]
		It is sufficient to prove the case $Z=X$, otherwise we can consider the subsystem $(Z,f)$. We only prove the case for the upper metric mean dimension, as the case for the lower metric mean dimension is similar.
		
		Firstly, we prove that $\overline{\operatorname{mdim}}_{\mathrm{M}}(X, f, d)\ge\overline{\operatorname{mdim}}_{\mathrm{M}}^B(X, f, d)$. For any $r>\overline{\operatorname{mdim}}_{\mathrm{M}}(X, f, d)$,  by the definition \eqref{1}, there exists $\epsilon_0>0$ such that for any $\epsilon\in(0,\epsilon_0)$, $h_{\operatorname{top}}(f,X,\epsilon)<r|\log\epsilon|$. Moreover, there exists $N_0>0$ such that for any $n>N_0$, $\frac{1}{n}\log s(f,X,n,\epsilon)<r|\log\epsilon|$, and so
		$s(f,X,n,\epsilon)<\exp(nr|\log\epsilon|)$. By the definition \eqref{def m Z,s,Nepsilon} and the fact that maximal $(n,\epsilon)$-separated is $(n,\epsilon)$-spanning,
		one has
		\[m(X,r|\log\epsilon|,n,\epsilon)\le s(f,X,n,\epsilon)\exp(-nr|\log\epsilon|)<1.\] Therefore, $m(X,r|\log\epsilon|,\epsilon)\le1$. As a consequence, $h_{\operatorname{top}}^B(X,f,\epsilon)\le r|\log\epsilon|$ for any $\epsilon\in (0,\epsilon_0)$ and hence $\overline{\operatorname{mdim}}_{\mathrm{M}}^B(X, f, d)\le r$. As $r>\overline{\operatorname{mdim}}_{\mathrm{M}}(X, f, d)$ is arbitrary, one obtains $\overline{\operatorname{mdim}}_{\mathrm{M}}(X, f, d)\ge \overline{\operatorname{mdim}}_{\mathrm{M}}^B(X, f, d)$.
		
		Secondly, to show $\overline{\operatorname{mdim}}_{\mathrm{M}}(X, f, d)\le\overline{\operatorname{mdim}}_{\mathrm{M}}^B(X, f, d)$, we need  some preparations, which can be found in \cite[Chap. 4, Sec. 11]{Pesin}. For any nonempty subset $Z\subset X$, $s\in\mathbb{R}$, $N\in\mathbb{N}$ and $\epsilon>0$, we  define
		\begin{equation}\label{def m appendix}
			M(Z, s, N, \epsilon)=\inf _{\Gamma}\left\{\sum_{i \in I} \exp \left(-sN\right)\right\},
		\end{equation}
		where the infimum is taken over all finite or countable covers $\Gamma=\left\{B_{N}\left(x_i, \epsilon\right)\right\}_{i \in I}$ of $Z$.
		Let
		\[\overline M(Z, s, \epsilon)=\limsup_{N \rightarrow \infty} M(Z, s, N, \epsilon),\ \underline M(Z, s, \epsilon)=\liminf_{N \rightarrow \infty} M(Z, s, N, \epsilon).\] Then the following quantities exist
		\begin{align*}
			\overline{CP}(Z,\epsilon)&:=\inf \{s\in\mathbb{R}: \overline M(Z, s, \epsilon)=0\}=\sup \{s\in\mathbb{R}: \overline M(Z, s, \epsilon)=\infty\},\\
			\underline{CP}(Z,\epsilon)&:=\inf \{s\in\mathbb{R}: \underline M(Z, s, \epsilon)=0\}=\sup \{s\in\mathbb{R}: \underline M(Z, s, \epsilon)=\infty\}.
		\end{align*}We define the upper and lower capacity metric mean dimension by
		\begin{align*}
			\overline{CPmdim}_{\mathrm{M}}(Z,f,d)=\limsup_{\epsilon\to 0}\frac{\overline{CP}(Z,\epsilon)}{|\log\epsilon|}
		\end{align*}and
		\begin{equation*}
			\underline{CPmdim}_{\mathrm{M}}(Z,f,d)=\liminf_{\epsilon\to 0}\frac{\underline{CP}(Z,\epsilon)}{|\log\epsilon|}.
		\end{equation*}
		For any continuous function $\varphi:X\to \mathbb{R}$, and a finite open cover $\mathcal{U}$, let $P_Z(\varphi,\mathcal{U})$, $\underline{CP}_Z(\varphi,\mathcal{U})$ and $\overline{CP}_Z(\varphi,\mathcal{U})$ be defined in \cite[Theorem 11.1]{Pesin}.  \cite[Theorem 11.5]{Pesin} shows that for any compact $f$-invariant set $Z$, one has
		\begin{equation*}\label{}
			P_Z(\varphi,\mathcal{U}) =\underline{CP}_Z(\varphi,\mathcal{U})=\overline{CP}_Z(\varphi,\mathcal{U}).
		\end{equation*}In particular, for $\varphi=0$, $Z=X$, we have
		\begin{equation}\label{A 2}
			P_X(0,\mathcal{U}) =\underline{CP}_X(0,\mathcal{U})=\overline{CP}_X(0,\mathcal{U}).
		\end{equation}We denote $P_X(\mathcal{U}):= P_X(0,\mathcal{U})$, $\underline{CP}_X(\mathcal{U}):=\underline{CP}_X(0,\mathcal{U})$, and $\overline{CP}_X(\mathcal{U}):=\overline{CP}_X(0,\mathcal{U})$ for short. The relationship (11.11) in \cite[Page 74]{Pesin} implies that
		\begin{align}
			\overline{CP}(X,2\operatorname{diam}(\mathcal{U})) &\leq \overline{CP}_X(\mathcal{U})\leq \overline{CP}(X,\frac{1}{2}Leb(\mathcal{U})),\label{A3}\\
			\underline{CP}(X,2\operatorname{diam}(\mathcal{U})) &\leq \underline{CP}_X(\mathcal{U})\leq \underline{CP}(X,\frac{1}{2}Leb(\mathcal{U})),\label{A4}
		\end{align}and
		\begin{equation}\label{A leq leq}
			h_{\operatorname{top}}^B(X,f,2\operatorname{diam}(\mathcal{U}))\leq P_X(\mathcal{U})\leq  h_{\operatorname {top}}^B(X,f,\frac{1}{2}Leb(\mathcal{U})).
		\end{equation}Now for any $\epsilon>0$, by Lemma \ref{lem4},
		there exists a finite open cover $\mathcal{U}_\epsilon$ of $X$ such that $\operatorname{diam}(\mathcal{U}_\epsilon)\leq \epsilon$ and $Leb(\mathcal{U}_\epsilon)\geq \frac{\epsilon}{4}.$ Now, on the one hand, we have
		\begin{align*}
			\overline{\operatorname{mdim}}_{\mathrm{M}}^B(X, f, d)  & =\limsup_{\epsilon\to 0}\frac{h_{\operatorname{top}}^B(X,f,\epsilon)}{|\log\epsilon|}=\limsup_{\epsilon\to 0}\frac{h_{\operatorname{top}}^B(X,f,2\epsilon)}{|\log2\epsilon|}\leq \limsup_{\epsilon\to 0}\frac{h_{\operatorname{top}}^B(X,f,2\operatorname{diam}(\mathcal{U}_\epsilon))}{|\log2\epsilon|}\\
			&\overset{\eqref{A leq leq}}\leq \limsup_{\epsilon\to 0}\frac{P_X(\mathcal{U}_\epsilon)}{|\log2\epsilon|}\overset{\eqref{A 2}}= \limsup_{\epsilon\to 0}\frac{\overline{CP}_X(\mathcal{U}_\epsilon)}{|\log2\epsilon|}\overset{\eqref{A3}}\leq \limsup_{\epsilon\to 0}\frac{\overline{CP}(X,\frac{1}{2}Leb(\mathcal{U}_\epsilon))}{|\log2\epsilon|}\\
			&\leq \limsup_{\epsilon\to 0}\frac{\overline{CP}(X,{\epsilon}/{8})}{|\log2\epsilon|}
			= \limsup_{\epsilon\to 0}\frac{\overline{CP}(X,\epsilon)}{|\log\epsilon|}=\overline{CPmdim}_{\mathrm{M}}(X,f,d).
		\end{align*}On the other hand, we also have
		\begin{align*}
			\overline{CPmdim}_{\mathrm{M}}(X,f,d) & =\limsup_{\epsilon\to 0}\frac{\overline{CP}(X,\epsilon)}{|\log\epsilon|}=\limsup_{\epsilon\to 0}\frac{\overline{CP}(X,2\epsilon)}{|\log2\epsilon|}\leq \limsup_{\epsilon\to 0}\frac{\overline{CP}(X,2\operatorname{diam}(\mathcal{U}_\epsilon))}{|\log2\epsilon|}\\
			&\overset{\eqref{A3}}\leq  \limsup_{\epsilon\to 0}\frac{\overline{CP}_X(\mathcal{U}_\epsilon)}{|\log2\epsilon|}\overset{\eqref{A 2}}=\limsup_{\epsilon\to 0}\frac{P_X(\mathcal{U}_\epsilon)}{|\log2\epsilon|}\overset{\eqref{A leq leq}}\leq \limsup_{\epsilon\to 0}\frac{h_{\operatorname{top}}^B(X,f,\frac{1}{2}Leb(\mathcal{U}_\epsilon))}{|\log2\epsilon|}\\
			&\leq  \limsup_{\epsilon\to 0}\frac{h_{\operatorname{top}}^B(X,f,{\epsilon}/{8})}{|\log2\epsilon|}=\limsup_{\epsilon\to 0}\frac{h_{\operatorname{top}}^B(X,f,\epsilon)}{|\log\epsilon|}=\overline{\operatorname{mdim}}_{\mathrm{M}}^B(X, f, d).
		\end{align*}Therefore, we obtain
		\begin{equation*}
			\overline{\operatorname{mdim}}_{\mathrm{M}}^B(X, f, d)=\overline{CPmdim}_{\mathrm{M}}(X,f,d).
		\end{equation*}
		Now for any $r> \overline{\operatorname{mdim}}_{\mathrm{M}}^B(X, f, d)=\overline{CPmdim}_{\mathrm{M}}(X, f,d)$, there exists $\epsilon_0>0$ such that for any $\epsilon\in(0,\epsilon_0)$, $\overline{CP}(X, \epsilon)< r|\log\epsilon|,$ which means that
		$\overline{M}(X,r|\log\epsilon|,\epsilon)=0$. So for any $\delta\in(0,1)$ there exists $N_0>0$ such that for any $N>N_0$, $M(X,r|\log\epsilon|,N,\epsilon)<\delta/2$. So, by definition \eqref{def m appendix}, there exists a finite cover $\Gamma=\{B_N(x_i,\epsilon)\}_{i\in I}$ such that
		\begin{equation}\label{21}
			\exp\left(-r|\log\epsilon|N\right)\cdot\#I<\delta.
		\end{equation}
		Note that $\{x_i\}_{i\in I}$ is a $(N,\epsilon)$-spanning set of $X$. Then
		$s(f,X,N,2\epsilon)$, the largest cardinality of all maximal $(N,2\epsilon)$-separated subset of $X$, is less than or equal to $\#I$. So by \eqref{21}, for any $\epsilon\in(0,\epsilon_0)$ and $N>N_0$,
		\begin{equation*}\label{20}
			\begin{split}
				\exp\left(-r|\log\epsilon|N\right)\cdot s(f,X,N,2\epsilon)<\delta.
			\end{split}
		\end{equation*}Note that $N>N_0$ is arbitrary, and  therefore we have
		\begin{equation*}
			h_{\operatorname{top}}(f,2\epsilon)=\lim_{n\to\infty}\frac{1}{n}\log s(f,X,n,2\epsilon)\leq r|\log\epsilon|,
		\end{equation*}which implies
		\begin{equation*}
			\overline{\operatorname{mdim}}_{\mathrm{M}}(X, f, d)=\limsup_{\epsilon\to 0}\frac{ h_{\operatorname{top}}(f,\epsilon)}{|\log\epsilon|}=\limsup_{\epsilon\to 0}\frac{ h_{\operatorname{top}}(f,2\epsilon)}{|\log2\epsilon|}\leq r.
		\end{equation*}Since $r> \overline{\operatorname{mdim}}_{\mathrm{M}}^B(X, f, d)$ is arbitrary, we obtain $\overline{\operatorname{mdim}}_{\mathrm{M}}(X, f, d)\leq \overline{\operatorname{mdim}}_{\mathrm{M}}^B(X, f, d).$
		The proof of Proposition \ref{propositiob coincides} is completed.
	\end{proof}

	\begin{proof}[Proof of Lemma \ref{lem5}]
		Given a finite open cover $\mathcal{U}$ of $X$, $\varphi\in C(X,\mathbb{R})$, a measure $\mu\in M(X,f)$, and any $\delta>0$.
		It is well known that the weak$^*$-topology on $M(X)$ is metrizable, and denote $d^*$ to be one of the compatible metrics. Let $\beta>0$ be sufficiently small such that for every $\tau_1, \tau_2 \in$ $M(X,f)$,
		$$
		d^*\left(\tau_1, \tau_2\right)<\beta \Longrightarrow\left|\int \varphi d \tau_1-\int \varphi d \tau_2\right|<\delta.
		$$
		Take  a partition $\mathcal{P}=\left\{P_1, \ldots, P_{j}\right\}$ of $M(X,f)$ whose diameter with respect to $d^*$ is smaller than $\beta$. By the ergodic decomposition theorem (see for example \cite[Page 153, Remark (2)]{Peter}) there exists a measure $\hat{\mu}$ on $M(X,f)$ satisfying $\hat{\mu}\left(M^e(X,f)\right)=1$ and
		$$
		\int \psi(x) d \mu(x)=\int_{M^e(X,f)}\left(\int_X \psi(x) d \tau(x)\right) d \hat{\mu}(\tau) \text { for every } \psi \in C(X, \mathbb{R}).
		$$
		By \eqref{eq lemma2.6}, we have
		\[\sup _{\tau \in M^e(X,f)} \inf_{\xi \succ \mathcal{U}} h_\tau(f, \xi)\leq \sup _{\tau \in M^e(X,f)}\underline{h}_{\tau}(f,Leb(\mathcal{U}),\delta')\leq h_{top}(f,Leb(\mathcal{U}))<\infty.\]
		Thus, we can pick $\nu_i \in P_i \cap M^e(X,f)$ such that
		\begin{equation*}
			\inf _{\xi \succ \mathcal{U}} h_{\nu_i}(f, \xi) \geq\inf _{\xi \succ \mathcal{U}} h_\tau(f, \xi)-\delta\mbox{ for $\hat{\mu}$-almost every $\tau \in P_i \cap M^e(X,f)$}.
		\end{equation*}
		\begin{comment}
			noticing \cite[Section 5]{HMRY} or \cite[Theorem 2]{S})
			\begin{equation*}
				\sup _{\tau \in M^e(X,f)} \inf_{\xi \succ \mathcal{U}} h_\tau(f, \xi)=\sup _{\tau \in M(X,f)} \inf_{\xi \succ \mathcal{U}} h_\tau(f, \xi)=h_{\operatorname{top}}(f,\mathcal{U}),
			\end{equation*}and by \cite[Theorem 4&5]{S}
		\end{comment}
		Let us consider $\lambda_i=\hat{\mu}\left(P_i\right)$ and define $\nu=\sum_{i=1}^j\lambda_i\nu_i$. We can check (1) and (3) in the statement of Lemma \ref{lem5} directly from the construction. For (2), by \cite[Proposition 5]{HMRY}, one has
		\[\inf_{\xi \succ \mathcal{U}}h_{\mu}(f,\xi)=\int_{M^e(X,f)}\inf_{\xi \succ \mathcal{U}}h_{\tau}(f,\xi)d\hat{\mu}(\tau).\]
		Thus, by the choice of the measure $\nu_i$, one has
		\[\inf_{\xi \succ \mathcal{U}}h_{\mu}(f,\xi)=\sum_{i=1}^j\int_{P_i\cap M^e(X,f)}\inf_{\xi \succ \mathcal{U}}h_{\tau}(f,\xi)d\hat{\mu}(\tau)\le \sum_{i=1}^j\lambda_i\inf _{\xi\succ \mathcal{U}} h_{\nu_i}(f, \xi)+\delta,\]
		which finishes the proof of Lemma \ref{lem5}.
	\end{proof}

	\section*{Acknowledgement}
	The authors would like to thanks   Professor Dou Dou, Professor Paulo Varandas and the referee for kindly suggestions.
	C. Liu was partially supported by NNSF of China (12090012, 12031019,  12090010). X. Liu was supported by NNSF of China (12090012, 12090010) and the China Postdoctoral Science Foundation (2022M723056).

	\bibliographystyle{acm}

\end{document}